\documentclass{amsart}
\usepackage[dvips]{color}
\usepackage{graphicx}

\usepackage{color}
\usepackage{latexsym}
\usepackage{amssymb}
\usepackage{amsmath}
\usepackage{amsthm}
\usepackage{amscd}
\theoremstyle{plain}
\newtheorem*{thm*}{Theorem}
\newtheorem{thm}{Theorem}[section]
\newtheorem{lem}[thm]{Lemma}
\newtheorem{defin}[thm]{Definition}
\newtheorem{prop}[thm]{Proposition}
\newtheorem{cor}[thm]{Corollary}
\newtheorem{slem}[thm]{Sublemma}

\theoremstyle{remark}
\newtheorem{rem}[thm]{Remark}

\newcommand{\LS}{\ensuremath{\underset{n=1}{\overset{\infty}{\cap}} \, {\underset{i}{\overset{\infty}{\cup}}}\,}}

\author[J.\ Chaika]{Jon Chaika}\email{chaika@math.utah.edu}\address{Department of Mathematics, University of Utah, 155 S 1400 E, Room 233, Salt Lake City, UT~84112, USA}
\author[D. Constantine]{David Constantine}\email{dconstantine@wesleyan.edu}\address{Department of Mathematics and Computer Science, Wesleyan University, 265 Church Street, Middletown, CT~06459, USA}

\begin{document}

\title[Quantitative shrinking targets]{Quantitative shrinking target properties for rotations and interval exchanges}
\maketitle

%

\section{Introduction}
Let $\alpha \in [0,1)$. The rotation $R_{\alpha}:[0,1) \to [0,1)$ by $R_{\alpha}(x)=x+\alpha \mod 1$ is one of the most natural and best understood dynamical systems. For example, Herman Weyl proved the following result on the asymptotic frequency with which an orbit visits a fixed ball:

\begin{thm*}
Let $\alpha\notin \mathbb{Q}$. Then for any $\epsilon>0$ and any $a\in[0,1)$ we have
	\[\lim_{N \to \infty}\frac{\sum_{i=1}^{N}\chi_{B(a, \epsilon)}(R_{\alpha}^ix)}{N2\epsilon}=1.\]
\end{thm*}

This paper concerns the following question: What if the ball's radius is allowed to shrink as $i$ increases? The focus of this paper is on treating families of sequences of radii $\{r_i\}$ simultaneously and obtaining explicit conditions on $\alpha$ under which theorems like the above can be proved. The following is the main result of this paper for rotations:

\begin{thm}\label{khinchin seq} 
There exists an explicit, full measure diophantine condition on $\alpha \notin \mathbb{Q}$ so that if $\alpha$ satisfies this condition then for any sequence $\{r_i\}$ such that $ir_i$ is non-increasing and $\sum_{i=1}^{\infty}r_i =\infty$, and for any $a\in [0,1)$ we have
 \begin{equation}\label{eqn:thm1}
 	\lim_{N \to \infty}\frac{\sum_{i=1}^N \chi_{B(a,r_i)}(R_{\alpha}^ix)}{\sum_{i=1}^N 2r_i}=1 
\end{equation}
for almost every $x$. 
\end{thm}

If $\alpha$ is badly approximable (a measure zero, full Hausdorff dimension set) then we can relax the condition on the radius sequences further:

\begin{thm} \label{constant type} 
If $\alpha$ is badly approximable, $\{r_i\}_{i=1}^{\infty}$ is non-increasing, and $\sum r_i= \infty$, then for any $a\in [0,1)$
\[\lim_{N \to \infty}\frac{\sum_{i=1}^N \chi_{B(a,r_i)}(R_\alpha^i x)}{\sum_{i=1}^N 2r_i}=1\] 
for almost every $x$.
\end{thm}

The choice of the center of these balls  $a$ does not play any role in our proof. For the sake of concreteness, outside of the statements of our theorems we will prove all our results for $a=\frac{1}{2}$. The full measure set of $x$ for which our theorems hold does, of course, depend on $a$.

We note that Kurzweil showed that the conclusion of Theorem \ref{constant type} can hold at most for badly approximable $\alpha$:

\begin{thm*} (Kurzweil \cite{kurz})\label{startthm} 
For any decreasing sequence of positive real numbers $\{r_i\}_{i=1}^{\infty}$ with divergent sum there exists $\mathcal{V} \subset [0,1)$, a full measure set of $\alpha$, such that for all ${\alpha \in \mathcal{V}} $ we have
\[ m\left(\LS B(R_{\alpha}^{-i}(x),r_i)\right)=1\]
 for every $x$, where $m$ denotes Lebesgue measure.
  
On the other hand,
\[ m\left(\LS B(R_{\alpha}^{-i}(x),r_i)\right)=1\]
for every $x$ and every decreasing sequence of positive real numbers $\{r_i\}_{i=1}^{\infty}$~  with divergent sum iff $\alpha$ is badly approximable.
\end{thm*}

Let us make the statements of Theorems \ref{khinchin seq} and \ref{constant type} precise.  We call a sequence $\{r_i\}$ where $ir_i$ is non-increasing and $\sum r_i=\infty$ a \emph{Khinchin sequence}.  Let $[a_1,...]$ be the continued fraction expansion of $\alpha$. The number $\alpha$ is \emph{ badly approximable }if $\underset{n \to \infty}{\limsup} \phantom{i}a_n<\infty$. The diophantine condition in Theorem \ref{khinchin seq} is as follows: 
 
 \begin{itemize}
	\item $a_n<n^{\frac{4}{3}}$ for all but finitely many $n$,
	\item $\underset{C \to \infty}{\lim}\underset{N \to \infty}{\limsup}\, \frac{1}{N} \left(\underset{i=1}{\overset{N}{\sum}}\log a_i-\underset{i:a_i<C}{\overset{N}{\sum}}\log a_i\right)=0,$ and
	\item $\underset{k:a_k>k^{\frac{1}{2}}}{\sum} \frac{\log k}{k^{\frac{2}{3}}} < \infty$. 
 \end{itemize}
Here, and throughout the paper, $\sum_{i\in S}^N$ means $\sum_{i\in S\cap[0,N]}$.

The first condition is a standard full measure condition on $\alpha$ (see, e.g., \cite[Thm 30]{khinchin}). 

The second is a mild ``non-divergence" condition. The $\alpha$ which satisfy it have full measure, which can be seen as follows. Let $\mu$ be the Gauss measure on $[0,1)$ and consider the $L^1(\mu)$ functions $\gamma(x) = \log( \lfloor \frac 1 x \rfloor)$ -- the logarithm of the first term in the continued fraction expansion of $x$ -- and 
\[\gamma_a(x)=\begin{cases} \log(a) & \text{ if }\lfloor \frac 1 x \rfloor =a\\0 &\text{ else}\end{cases}.\]
 Applying the Birkhoff Ergodic Theorem for the Gauss map, $\phi(x)=\frac{1}{x} - \lfloor \frac{1}{x}\rfloor$ to $\gamma(x)- \sum_{a=1}^{C-1} \gamma_a(x)$ and noting that $\Vert \gamma(x) - \sum_{a=1}^{C-1} \gamma_a(x) \Vert_1 \to 0$ as $C\to \infty$ gives the result.
 
For the third condition, recall that $m(\{ \alpha: a_j(\alpha)=k \}) < \frac{2}{k^2}$ (see, e.g., \cite[p. 60]{khinchin}) so $m(\{\alpha : a_j(\alpha) \geq k \}) \leq \frac{D}{k}$ for some constant $D$. Let $f_j(\alpha) = \frac{\log j}{j^{2/3}} \chi_{\{ a_j> j^{1/2}\}}(\alpha)$. Then $\int f_j dm \leq \frac{\log j}{j^{2/3}} \frac{D}{j^{1/2}}.$ Letting $g(\alpha) = \sum_{j=1}^\infty f_j(\alpha)$, since the integrals of the $f_j$ are summable, $\int g(\alpha) dm = \sum_{j=1}^\infty \int f_j(\alpha) dm <\infty$. Therefore $g$ is finite for almost all $\alpha$, that is, our third condition holds for almost all $\alpha$.

We will prove our results not just for rotations, but also for interval exchange transformations (IETs; Definition \ref{iet def}) satisfying similar diophantine assumptions. The statement of this more general theorem (Theorem \ref{thm:explicit}) requires a few technical definitions and so is delayed until Section \ref{sec:Thm1}. We mention D. Kim and S. Marmi \cite{log iet}, S. Galatolo \cite{galatolo}, L. Marchese \cite{march}, M. Boshernitzan and J. Chaika \cite{gauges}, M. Marmi, S. Moussa and J-C. Yoccoz \cite{mmy} where a variety of diophantine results for interval exchanges and rotations are proven.

A key tool in extending our work to IETs is a quantitative version of Boshernitzan's criterion for unique ergodicity which may be of independent interest (see Section \ref{sec:quant bc} for terminology, historical discussion and proof). We call an interval bounded by two adjacent discontinuities of $T^n$ (counting 0 and 1 as discontinuities) an \emph{$n$-block interval} of $T$ (see Definition \ref{defn:n block interval} in the Appendix).

\begin{thm} \label{thm:quant bc}  
Let $T$ be a minimal interval exchange transformation. Let $e_T(n)$ denote the minimum measure of any $n$-block interval of $T$. Let $c>0$. Assume $n_j\in \mathbb{N}$ have the following two properties:
\begin{enumerate}
	\item $\frac{n_{j+1}}{n_j}>2$
	\item $e_T(n_j)>\frac {c}{n_j}$ for all $j$.
\end{enumerate}
Let $J$ be any $n_i$-block interval of $T$. Then there exist constants $C_1,C_2,\hat{q}>0$ depending only on $c$ such that for any points $x,x'$ we have 
\[\frac 1 {n_{i+\hat{q}+L}}\Big|\underset{j=1}{\overset{n_{i+\hat{q}+L}}{\sum}} \chi_J(T^jx)-\chi_J(T^jx')\Big|<C_1e^{-C_2L}|J|\]
for all $L\in \mathbb{N}$. $|J|$ denotes the length of $J$.
\end{thm}

Quantitative equidistribution results for interval exchanges have also been proven by A. Zorich \cite{zorich}, G. Forni \cite{forni}, and J. Athreya and G. Forni \cite{ath_forn}.

%

\subsection{Related results in other settings}

\begin{defin} Given a dynamical system $(X,T,\mu)$, a sequence of sets $\{C_i\}$ is a \emph{strong Borel Cantelli} sequence for $T$ if 
$$\underset{N \to \infty}{\lim}\, \frac{\sum_{i=1}^N\chi_{C_i}(T^ix)}{\sum_{i=1}^N\mu(C_i)}=1$$
for almost every $x$.
\end{defin}

This paper establishes that for almost every $\alpha$, any sequence of balls $B(\frac 1 2 ,r_i)$  so that $\{r_i\}$ is a Khinchin sequence is strong Borel Cantelli for $R_\alpha$. If the rotation is badly approximable we may relax the condition to allow $r_i$ just non-increasing and with divergent sum. 

This question has been considered in systems of high complexity. Philipp \cite{philipp} proved that for the Gauss map, or a $\beta$-shift with the smooth invariant measure any sequence of intervals so that the sum of the measures diverge is strong Borel Cantelli. Dolgopyat \cite{dolg} proved an analogous result for Anosov diffeomorphisms. Chernov-Kleinbock \cite{sbc} proved a similar result for topological Markov chains with a Gibbs measure: cylinders satisfying a certain nesting condition and so that the sum of their measures diverge are strong Borel Cantelli.  To highlight the difference between our low complexity setting and the high complexity situation we remark that for every rotation $\alpha$ there is a sequence of sets $\{C_i\}$ with each $C_i \in \{[0,1], [\frac 1 4, \frac 3 4 ]\}$ which is not strong Borel Cantelli. 

%

\subsection{Outline of paper}

We prove our results following a proof of the strong law of large numbers. 

In Section \ref{sec:Thm1}, we prove Theorem \ref{thm:explicit}, the generalization of Theorem \ref{khinchin seq} to IETs. The first key step is Proposition \ref{prop:correlation}, which we prove in Section \ref{indep}. This Proposition says that, in the presence of the diophantine assumption, a large part of the sum in the numerator of equation \eqref{eqn:thm1} can be broken up into sums over disjoint ranges for $i$ in such a way that the resulting quantities are approximately independent.

Section \ref{sec:abstract} shows, via this approximate independence result, that Theorem \ref{khinchin seq} is true if we ignore those terms in the sum which are not part of these roughly independent quantities. Then Section \ref{sec:limsup} treats the terms ignored in Section \ref{sec:abstract}, showing that their contribution is negligible and finishing the proof.

We then prove Theorem \ref{constant type} in two parts. In Section \ref{sec:reasonable} we treat radius sequences $\{r_i\}$ where $\sup\, ir_i<\infty$. In Section \ref{sec:big section} we treat the general case.

Section \ref{sec:quant bc} proves the quantitative Boshernitzan criterion, Theorem \ref{thm:quant bc}, which is used in the earlier sections.

There is an appendix that  provides a treatment of the symbolic coding of an IET. This is well-known material included for completeness, and to provide a reference for notation and terminology used elsewhere in the paper.

\subsection{Acknowledgments}
J. Chaika would like to thank B. Fayad and D. Kleinbock for encouraging me to pursue this question. We would like to thank J. Athreya, M. Boshernitzan, A. Eskin, H. Masur, R. Vance and W. Veech for helpful conversations. J. Chaika was partially supported by NSF grants DMS-1004372, DMS-135500 and DMS-1452762, a Sloan fellowship and a Warnock chair. We are also deeply indebted to anonymous referees for many helpful suggestions on earlier versions of the paper.

%

\section{Proof of Theorem \ref{khinchin seq}}\label{sec:Thm1}
\subsection{Setup and an outline of the proof}\label{sec:setup}

In this section we introduce notation and terminology necessary to state and prove Theorem \ref{thm:explicit} -- our extension of Theorem \ref{khinchin seq} to interval exchange transformations. Our first task is to introduce an analogue of the continued fraction expansion used to state Theorem \ref{khinchin seq}. We also give a short outline of the proof of Theorem \ref{thm:explicit} and record a few lemmas for future use.

\begin{defin} \label{iet def} Given a vector $L=(l_1,l_2,...,l_d)$ where $l_i \geq 0$ and $\sum_{i=1}^d l_i =1$, we obtain $d$ sub-intervals of $[0,1)$: 
\[I_1=[0,l_1), \ I_2=[l_1,l_1+l_2), \ ..., \ I_d=[l_1+...+l_{d-1},1).\] 
Given a permutation $\pi$ on  the set $\{1,2,...,d\}$, we obtain a d-\emph{Interval Exchange Transformation} (IET)  $T \colon [0,1) \to [0,1)$ which exchanges the intervals $I_i$ according to $\pi$. That is, if $x \in I_j$ then 
\[T(x)= x - \underset{k<j}{\sum} l_k +\underset{\pi(k')<\pi(j)}{\sum} l_{k'}.\]
\end{defin}

Throughout the paper, we work with the Lebesgue measure on $[0,1)$, which is invariant under any IET. The Lebesgue measure of a set $A$ will be denoted by $m(A)$. For intervals, we will write $|J|$ for $m(J)$.

The points $D=\{\sum_{i=1}^rl_i\}_{r=1}^{d-1}$ are the \emph{discontinuities} of $T$. The discontinuities of $T^n$ are $\bigcup_{i=0}^{n-1}T^{-i}D$. Generalizing the behavior of irrational circle rotations to IETs is the Keane condition:

\begin{defin}
$T$ satisfies the \emph{Keane condition} if the orbits of all its discontinuities are infinite and disjoint.
\end{defin}
This full measure condition will be assumed for Theorems \ref{thm:explicit} and \ref{thm:iet bad approx}.

Given an IET $T$, let $e_T:\mathbb{N}\to \mathbb{R}$ be defined as follows: $e_T(n)$ is the minimum distance between two discontinuities of $T^n$. If two discontinuities orbit into each other then $e_T(n)$ is defined to be 0. Since $T^{-1}(\{0,1\})$ is contained in the set of discontinuities we have that $e_T(n)$ is at most (i.e. $\leq$\footnote{Throughout the paper, when we write `at most' we mean $\leq$. We avoid using this term when a distinction between $<$ and $\leq$ is important for our arguments.}) the measure of the smallest $(n-1)$-block interval (see Appendix \ref{symb code}). Notice that $e_T$ is a non-increasing function.

Fix $\xi>0$. We define an increasing sequence of integers $n_i(\xi)$ inductively as follows. Let $n_0(\xi)=1$ and let $n_{i+1}=\min\{2^k>n_i: e_T(2n_{i+1})>\frac{\xi}{n_{i+1}} \}$. Let $a_{i}(\xi)=\frac{n_{i}}{n_{i-1}}.$ Below, we will suppress $\xi$ in our notation.

 \begin{thm}\label{thm:explicit} Let $T$ be an IET satisfying the Keane condition so that for every $\epsilon>0$ there exists $\xi>0$, $C$ so that
\begin{enumerate}
	\item $a_i\leq i^{\frac{4}{3}}$ for all but finitely many $i$,  
	\item $\limsup_{N \to \infty} \frac{1}{N} \left(\sum_{i=1}^N \log a_i-\sum_{a_i<C}^{N}\log a_i \right)<\epsilon,$ and
	\item $\sum_{k:a_k>k^{\frac{1}{2}}} \frac{\log k}{k^{\frac{2}{3}}} < \infty.$
\end{enumerate}
Then for any Khinchin sequence $\{r_i\}$ and any $a\in[0,1)$ we have 
\begin{equation}\label{eqn:Thm4}
	\lim_{N \to \infty}\frac{\sum_{j=1}^N \chi_{B(a,r_j)}(T^jx)}{\sum_{j=1}^N 2r_j}=1
\end{equation}
for almost every $x$.
\end{thm}

\begin{rem}
Note that for any $\xi<1$, for all rotations $R_\alpha$, if $q_i$ is the denominator of the $i^{th}$ convergent to $\alpha$, then $e_{R_\alpha}(q_i-1) > \frac{\xi}{2q_i}$ (see, e.g., \cite[\S6]{khinchin}). Since $q_{i+1}=a_i q_i + q_{i-1}$ where $\alpha = [a_1, a_2, \ldots]$ is the continued fraction expansion of $\alpha$, our definition of the $a_i$ for Theorems \ref{thm:explicit} and \ref{thm:iet bad approx} (the IET version of Theorem \ref{constant type}) is inspired by the partial fraction expression for rotations.

To see how Theorems \ref{khinchin seq} and \ref{constant type}, stated as they are for rotations, follow from Theorems \ref{thm:explicit} and \ref{thm:iet bad approx} stated for IETs, it is a short exercise to verify that if $2^k < q_j \leq 2^{k+1}$ then $2^{k-1}$ satisfies the inequality required to be an $n_i$ with $\xi\leq \frac{1}{4}$. Therefore, when addressed with the machinery of Theorems \ref{thm:explicit} and \ref{thm:iet bad approx}, new $n_i$ appear whenever a new $q_j$ is reached (except when $a_j=1$, in which case we may have to wait for $q_{j+1}$). The $a_i$ for the IET machinery will be bounded above by a small fixed multiple of the $a_j$ for the continued fraction expansion. Theorems \ref{thm:explicit} and \ref{thm:iet bad approx} still work after accounting for this multiple, proving Theorems \ref{khinchin seq} and \ref{constant type}.
\end{rem}

The proof of Theorem \ref{thm:explicit} proceeds as follows. First, we split up the sum in the numerator of equation \eqref{eqn:Thm4} into sums over disjoint sets of indices. Specifically, let 
\[g_i(x)=\sum_{j=n_i}^{2n_i-1}\chi_{B(\frac 1 2,r_j)}(T^jx).\]
These sums account for much, but not all, of the numerator in equation \eqref{eqn:Thm4}. In Sections \ref{indep} and \ref{sec:abstract} we show that Theorem \ref{thm:explicit} holds if we ignore terms not included in the $g_i$:
\begin{equation}\label{gi result}
	\underset{N \to \infty}{\lim}\frac{\underset{i=1}{\overset{N}{\sum}} g_i(x)}{\underset{i=1}{\overset{N}{\sum}} \int g_i}=1.
\end{equation}

\begin{rem}
Note that throughout the paper, all integrals are taken with respect to the Lebesgue measure on $[0,1]$.
\end{rem}

We prove equation \eqref{gi result} by showing that the $g_i$ satisfy the following version of the strong law of large numbers. Its (standard) proof is included in Section \ref{sec:abstract} for completeness.

\begin{prop}\label{prop:abstract crit} Let $H_{i}:[0,1] \to \mathbb{R}_{\geq 0}$ so that for all $i$ there exists $C_1,C_2$:

\begin{itemize}
	\item[(H1)] $\Vert H_i \Vert_{\infty}<C_1$
	\item[(H2)] $\sum_{i=1}^{\infty} \int H_i=+\infty$
	\item[(H3)] $\sum_{j=i+1}^{\infty} \left|\int H_j(x)H_i(x)-\int H_i(x)\int H_j(x) \right|<C_2 \Vert H_{i-1}(x)\Vert_1$.
\end{itemize}
Then 
\[\underset{N \to\infty} {\lim} \frac{\sum_{i=1}^NH_i(x)}{\sum_{i=1}^N\int H_i(x)}=1\]
for a.e. $x$.
\end{prop}

Property (H3) should be thought of as approximate independence of the $H_i$. Verifying it for $g_i$ is the main work; this is shown in Section \ref{indep}. This approximate independence for $g_i$ comes via Lemma \ref{lem:correl by invar} from an effective equidistribution result on $T$ (Theorem \ref{thm:quant bc}) and approximate $T$-invariance of the $g_i$ (Lemma \ref{lem:almost invar}).

Having established equation \eqref{gi result}, we complete the proof in Section \ref{sec:limsup} by showing that those times not accounted for by the $g_i$ contribute negligibly to equation \eqref{eqn:Thm4}. Let
\[\beta_i(x)=\sum_{j=2n_i}^{n_{i+1}-1}\chi_{B(\frac 1 2,r_j)}(T^jx).\]
We will prove that, for almost every $x$
\[\sum_{i<N}\beta_i(x)=o\left(\sum_{i=1 }^{n_N}\chi_{B(\frac 1 2, r_i)}(T^ix)\right).\]
The $\beta_i$ depend strongly on the $a_i$, and hence on the parameter $\xi$. When we need to make this dependence explicit (for example in the proof of Lemma \ref{lem:translation}), we will write $\beta_i(x,\xi)$ for $\beta_i(x)$.

Before proceeding to the main elements of the proof, we collect a few Lemmas we will need throughout. The first is a straightforward consequence of the Khinchin condition.

\begin{lem}\label{lem:L1 bound}
For any  $n\geq m$, 
\[ \sum_{i=n}^{2n-1} 2r_i \leq \sum_{i=m}^{2m-1} 2r_i.\]
In particular, for all $i\geq j$,
\[ \Vert g_i\Vert_1 \leq \Vert g_j \Vert_1. \]
\end{lem}

We conclude this section with a result used to control $\|g_i\|_\infty$ which we will frequently quote. Since the $\{r_i\}$ are a Khinchin sequence, this Lemma proves that the $g_i$ satisfy property (H1) from Proposition \ref{prop:abstract crit}.

\begin{lem}\label{lem:hit bound}  
$g_i(x)\leq 1+\frac {2n_i} {\xi} 2r_{n_i}$ for all $i$ and $x$.
\end{lem}

The proof relies on:

\begin{lem} \label{interval} (Boshernitzan \cite[Lemma 4.4]{rank2}) 
If $T$ satisfies the Keane condition, then for any interval $J$ with measure $\leq e_T(n+1)$ there exist integers ${p \leq 0 \leq q}$ (which depend on $J$) such that

\begin{enumerate}
	\item[(1)] $q-p \geq n$
	\item[(2)] $T^i$ acts continuously on $J$ for $p \leq i \leq q$
	\item[(3)] $T^i(J) \cap T^j(J)= \emptyset$ for $p \leq i < j \leq q$.
\end{enumerate}
\end{lem}

\begin{rem}
Boshernitzan proves a somewhat stronger result. One can remove the Keane condition assumption and get the same result as long as $J$ does not contain any saddle connections of $T$ (points on the orbit of two distinct discontinuities). The Keane condition implies that there are no saddle connections. 
\end{rem}

\begin{rem}\label{rem:note}
Note that condition (3) implies that $T^i(J) \cap T^j(J)=\emptyset$ for any interval with measure $\leq e_T(n+1)$ and $0<|i-j|\leq n$. 
\end{rem}

\begin{proof}[Proof of Lemma \ref{lem:hit bound}] 
Let $J$ be any interval. By Remark \ref{rem:note}, if $T^jx, T^{j+r}x \in J$, then $|J| > e_T(r+1)$.

Note that for all $x$, 
\[ g_i(x) = \sum_{j=n_i}^{2n_i-1} \chi_{B(\frac{1}{2}, r_j)}(T^jx) \leq  \sum_{j=n_i}^{2n_i-1} \chi_{B(\frac{1}{2}, r_{n_i})}(T^jx). \]
 Partition $B(\frac{1}{2}, r_{n_i})$ into subintervals $J_k$ with measure $\leq e_T(n_i+1)$, using as few intervals as possible. There are $\lceil \frac{2r_{n_i}}{e_T(n_i+1)} \rceil$ such intervals. Since the measure of $J_k$ is $\leq e_T(n_i+1)$, by Lemma \ref{interval}, if $n_i\leq j_1 < j_2 \leq 2n_i-1$, then at most one of $T^{j_1}x$ and $T^{j_2}x$ can lie in $J_k$. Hence
 \[g_i(x) \leq \left\lceil \frac{2r_{n_i}}{e_T(n_i+1)} \right\rceil \leq \frac{2r_{n_i}}{e_T(n_i+1)}+1\leq \frac{2r_{n_i}}{e_{T(2n_i)}}+1\leq \frac{2n_i2r_{n_i}}{\xi}+1.\]
\end{proof}

\begin{rem}\label{rem:no khinchin}
Note that the proof of Lemma \ref{lem:hit bound} uses only that we have a lower bound on $e_T(2n_i)$; the Khinchin condition does not play a role. Its argument will also extend to the setting of Proposition \ref{prop:reasonable i}, and we will use this argument in the proof of Lemma \ref{lem:large a}, using a different lower bound on $e_T$.
\end{rem}

%

\subsection{Estimate on $\mathbf{\int g_i(x)g_j(x)}$}\label{indep}

The goal of this section is to establish property (H3) of Proposition \ref{prop:abstract crit} for the $g_i$.

\begin{prop}\label{prop:correlation} 
There exists $C$ so that for all $j$,
\[\sum_{i=j+1}^{\infty}\left|\int g_ig_j-\int g_i \int g_j \right|<C \Vert g_{j-1}\Vert_1,\] 
where $C$ depends only on $\xi$.
\end{prop}

This proposition asserts `approximate independence' of $g_i$ and $g_j$ as $i$ becomes much larger than $j$. To prove this, when $i$ is sufficiently larger than $j$, we $L^1$-approximate $g_j$ by a function $f_{i,j}$ which is nearly independent from $g_i$. This function will be built using a general result, Lemma \ref{lem:const on blocks}, and a result using the dynamics of $T$, Lemma \ref{lem:dense block}, and it will be constant on certain intervals of $[0,1)$ closely related to the dynamics of $T$. Our first result shows how we can use a property of $g_i$ -- approximate $T$-invariance, established in Lemma \ref{lem:almost invar} -- to prove that $g_i$ is nearly independent from a function like $f_{i,j}$.

\begin{lem}\label{lem:correl by invar}
Assume $h$ is a non-negative function satisfying $\Vert h-h \circ T^i\Vert_1<\delta$ for $i\leq n$ and that $J$ is an interval such that
\[ \left|n|J|-|\{0<i\leq n:T^i(x)\in J\}|\right|<n\delta'.\]
Then
\[\left|\int h \chi_J-|J|\int h \right|\leq \delta'\left(\int h\right)+ \frac{1}{n} \sum_{i=1}^n \int_J |h(x)-h\circ T^i(x)| dx \leq \delta'\left(\int h\right)+\delta .\]
\end{lem}

\begin{proof}
Let $e_i(x)=h(x)-h\circ T^i(x)$. Then $\Vert e_i \Vert_1 < \delta$ for $i\leq n$. We have
\begin{align} 
	\int h(x)\chi_J(x)dx &= \int \frac 1 n \sum_{i=1}^n (h \circ T^i(x)+e_i(x))\chi_{J}(x)dx \nonumber \\
				& \leq \int \frac 1 n \sum_{i=1}^n h \circ T^i(x) \chi_J(x)dx+  \frac{1}{n} \sum_{i=1}^n \int_J |e_i(x)| dx \nonumber \\ 
				& = \int h(x)\frac 1 n \left|\{1\leq i\leq n:T^{-i}(x)\in J\}\right| dx+ \frac{1}{n} \sum_{i=1}^n \int_J |e_i(x)| dx  \nonumber \\
				& \leq (|J|+\delta')\left(\int h(x)dx\right)+\frac{1}{n} \sum_{i=1}^n \int_J |e_i(x)| dx. \nonumber
\end{align}
A similar calculation bounds $\int h \chi_J - |J| \int h$ below. The result follows from this and the bound on $\Vert e_i \Vert_1$.
\end{proof}

We want to apply Lemma \ref{lem:correl by invar} when $h=g_i$. To obtain good bounds, the first step is to establish the approximate $T$-invariance of $g_i$.

\begin{lem} \label{lem:almost invar} 
There exists $C$ so that for every $j$
\[\sum_{k=j+1}^{\infty}\max \{\Vert g_k-g_k \circ T^s\Vert_1:0\leq s<n_{\frac{k+j}2}\}<C \Vert g_{j-1}\Vert_1.\]
\end{lem}

\begin{proof}
For any $M<N$ and $0\leq s < N-M$,
\begin{align}
	\sum_{i=M}^N\chi_{T^{-i}B(\frac {1}{2},r_i)}(x)-&\sum_{i=M}^N\chi_{T^{-i}B(\frac{1}{2},r_i)}(T^sx) \nonumber \\
		& = 	\sum_{i=M}^N\chi_{T^{-i}B(\frac{1}{2},r_i)}(x)-\sum_{i=M}^N\chi_{T^{-i-s}B(\frac{1}{2},r_i)}(x) \nonumber \\
		& = 	\sum_{i=M}^N\chi_{T^{-i}B(\frac{1}{2},r_i)}(x)-\sum_{j=M+s}^{N+s}\chi_{T^{-j}B(\frac{1}{2},r_{j-s})}(x). \nonumber 
\end{align}
We now split the sums above into three parts -- the first $s$ terms of the first sum, the last $s$ terms of the second sum, and the middle terms where the indices in the two sums overlap. After converting the last $s$ terms back to the index $i$, this gives
\begin{align}		
	\sum_{i=M}^N\chi_{T^{-i}B(\frac {1}{2},r_i)}(x)-&\sum_{i=M}^N\chi_{T^{-i}B(\frac{1}{2},r_i)}(T^sx) \nonumber \\
		& = \sum_{i=M}^{M+s-1}\chi_{T^{-i}B(\frac{1}{2} ,r_i)}(x)-\sum_{i=N-s+1}^N\chi_{T^{-i-s}B(\frac{1}{2} ,r_i)}(x) \nonumber \\
		& \ \ \ + \sum_{i=M+s}^{N}\chi_{T^{-i}B(\frac{1}{2} ,r_i)}(x)-\chi_{T^{-i}B(\frac{1}{2} ,r_{i-s})}(x). \nonumber
\end{align}
Since we assume that $r_i$ is non-increasing, the $L_1$ norms of the first two terms are each bounded above by $2sr_M$. The $L_1$ norm of the third term can be bounded using a telescoping sum argument.  All but $2s$ terms cancel, giving a maximum total $L_1$ norm of $4sr_M$.

As $g_k-g_k\circ T^s$ has the above form with $M=n_k$ and $N=2n_k-1$, we prove the desired bound by bounding $8sr_M$ appropriately, then summing over $k\geq j+1$. By definition, $s \leq n_{\lfloor \frac{k+j}{2}\rfloor}$. Therefore, to prove the Lemma, it suffices to bound $\sum_{k=j+1}^\infty 8 n_{\lfloor \frac{k+j}{2} \rfloor} r_{n_k}$.  By construction, $n_{i+1} \geq 2 n_i$ so $n_k \geq 2^{\frac{k-j}{2}} n_{\lfloor \frac{k+j}{2} \rfloor}$ and hence $n_{\lfloor \frac{k+j}{2} \rfloor} \leq \frac{n_k}{\sqrt 2^{k-j}}$. Therefore,
\[ \sum_{k=j+1}^\infty 8 n_{\lfloor \frac{k+j}{2} \rfloor} r_{n_k} \leq 8 \sum_{k=j+1}^\infty \frac{1}{\sqrt 2^{k-j}} n_k r_{n_k}.\]
Since $n_kr_{n_k}$ is non-increasing, for some constant $C$ we have
\[ 8 \sum_{k=j+1}^\infty \frac{1}{\sqrt 2^{k-j}} n_k r_{n_k}\leq 8n_jr_{n_j}  \sum_{k=j+1}^\infty \frac{1}{\sqrt 2^{k-j}} \leq C \|g_{j-1}\|_1.\] \end{proof}

Essentially the same proof also demonstrates the following:

\begin{lem} \label{lem:beta almost invar}
There exists $C$ so that for every $j$
\[\sum_{k=j+1}^{\infty}\max \{\Vert \beta_k-\beta_k\circ T^s\Vert_1:0\leq s<n_{\frac{k+j}2}\}<C \Vert g_{j}\Vert_1.\]
\end{lem}

Theorem \ref{thm:quant bc} establishes convergence of orbit sums for functions that are constant on intervals of continuity of $T^M$ for appropriately chosen $M$. To use this with the $g_i$, we will need that some such function is close to $g_i$. The next two lemmas show this.

Given a finite set $S\subset [0,1]$ let $P_S$ be the finite partition of $[0,1]$ defined by connected components of $[0,1]\setminus S$. As a convention for this partition, we include each point in $S$ itself in the interval to its right (except, of course, when $1\in S$); this particular choice is not important in our arguments.

\begin{lem}\label{lem:const on blocks} 
If $S$ is $\epsilon$-dense then there exists a function $h$ which is constant on each element of $P_S$ and whose $L^1$ difference from $g_i$ is at most $2n_i\epsilon$. Moreover, $h$ can be chosen so that $\|h\|_\infty\leq \|g_i\|_\infty$, $\|h\|_1\leq \|g_i\|_1$, and $h$ can be expressed as the sum of $n_i$ characteristic functions for intervals.
\end{lem}

\begin{proof}
For any interval $J$, there exists some function $\phi$ which is constant on the elements of $P_S$ and such that $\Vert \chi_J - \phi\Vert_1 < 2\epsilon$ and $\Vert \phi\Vert_\infty \leq \Vert \chi_J \Vert_\infty$. Specifically, for each $I\in P_S$, if $I\subset J$, set $\phi=1$ on $I$, otherwise set $\phi=0$ on $I$. Note that $\phi$ is the characteristic function for an interval. The lemma follows because $g_i$ is the sum of $n_i$ characteristic functions of intervals.
\end{proof}

Let $S_{k}$ be the set of  discontinuities of $T^k$. Recall that $d$ is the number of intervals of our IET and that $e_T(n_i)>\frac \xi {2 n_i}$.

\begin{lem}\label{lem:dense block}
$S_{n_{i+d(2-\log_2(\xi))}}$ is $\frac{1}{n_i}$-dense.
\end{lem}

The lemma follows from the following result, which is adapted to our situation. This result uses the \emph{first return map}. Recall that if $G:X \to X$ is a dynamical system and $A \subset X$ then the first return map of $G$ to $A$ is $G|_A:A\to A$ by $G|_A(x)= G^{\min\{\ell>0:G^\ell x \in A\}}(x)$. The numbers $\min\{\ell>0:G^\ell x \in A\}$ are called return times. Recall that the first return map of a $d$-IET to an interval $J$ bounded by adjacent discontinuities of $T^n$ is a $\hat d$-IET for $\hat d\leq d$ and the return time is constant on each interval.

\begin{slem}\label{slem:dense} 
Let $J$ be an $m$-block interval of the $d$-IET $T$. Then at most $d(2-\log_2(\xi))$ of the  $n_i$ satisfy $\frac{1}{|J|}\leq n_i \leq m$.
\end{slem}

\begin{proof}
We assume $\frac{1}{|J|} < m$, as otherwise the statement is trivial.

\textit{Step 1:}  It suffices to show that there exist integers $k_1 \leq \cdots \leq k_{d'}$ with $d'\leq d$ such that $k_1<\frac{1}{|J|}$ and $k_{d'-1}< m \leq k_{d'}$ and such that if $k_j<i<k_{j+1}$ then $e_T(i)<\frac {1}{k_{j+1}}$. To see this,  say $n_l, \ldots, n_{l+c}$ lie in $[k_j, k_{j+1}]$. By the defining condition on the $n_l$ and the condition above, we have $\frac{\xi}{2n_l}<\frac{1}{k_{j+1}}$. Since the $\frac{n_{i+1}}{n_i}\geq 2$, we have 
\[2^cn_l \leq n_{l+c} \leq k_{j+1} < \frac{2}{\xi}n_l.\]
Hence $c < 1-\log_2\xi$ and so at most $2-\log_2\xi$ of the $n_i$ lie in $[k_j,k_{j+1}]$. Since at most $d$ intervals of this form cover $[\frac{1}{|J|},m]$, the result follows.

\textit{Step 2:} Defining a sequence. Consider the sequence of first return times $k_1 \leq \cdots \leq k_{d'}$ for the first return map $T|_J$. We note that the Keane condition guarantees that there must be a return time $\geq m$. To see this, recall that the Keane condition implies minimality, and then examine a point in the interior of $J$ which takes the minimum time to hit a discontinuity of $T$ (besides a discontinuity at an endpoint of $J$). This must occur before returning to $J$ (by the choice of the point) but also at time $\geq m$ (as $J$ is an $m$-block interval).

\textit{Step 3:} The sequence we defined satisfies the sufficient condition in Step 1.  Note that $J=[T^{-K}\delta,T^{-L}\delta')$ where $\delta,\delta'$ are either $0,1$ or discontinuities of $T$. Moreover, for any discontiuity $\delta''$, $T^{-r}\delta''\in Int(J)$ (the interior of $J$) implies $r\geq m$ since $J$ is an $m$-block. Write $I_i$ for the interval with return time $k_i$.

It is clear that the smallest return time, $k_1$, satisfies $k_1 < \frac 1 {|J|}$. Since $k_1<m$, by the remark above, the boundary point of $I_1$ in $Int(J)$ must be in the orbit of $\delta,\delta'$, because $T^i$ acts continuously on $J$ for $0\leq i<m$. Therefore, it is either $T^{-k_1-K}\delta$ or $T^{-k_1-L}(\delta')$. Without loss of generality, let us assume it is $T^{-k_1-K}\delta$. Pushing $J$ forward by $T^K$, we see that $T^{k_1+K}J$ intersects $T^KJ$. Let $K_1$ be the subinterval of $T^KJ$ which returns to $J$ after $k_1$ iterates of $T$ and $K_0$ the other subinterval. Note that $K_0$ is an $k_1$-block, and so $e_T(k_1)$ is bounded above by its length. Since $T^{-K}K_0$ has not returned to $J$ after $k_1$ iterates, the first return time for any of its points is $k_2$. As above, this implies that $|K_0|<\frac{1}{k_2}$. Therefore $e_T(k_1)<\frac{1}{k_2}$.

The argument above may be continued inductively, considering always the points which have not yet returned to $J$, as long as $k_i<m$. Therefore we have constructed the desired integers $k_i$ and the sublemma is proved.
\end{proof}
\begin{proof}[Proof of Lemma \ref{lem:dense block}]
Let $m = n_{i+d(2-\log_2(\xi))}$ and suppose, towards a contradiction, that $S_{m}$ is not $\frac 1 {n_i}$-dense and so there exists an $m$-block interval $J$ with $|J|\geq \frac{1}{n_i}$. Then $\frac{1}{|J|}\leq n_i<m$. Applying Sublemma \ref{slem:dense}, there can be at most $d(2-\log_2(\xi))$ of the $n_j$ between $\frac{1}{|J|}$ and $m$ (inclusive). But note that $n_i\geq \frac 1 {|J|}$ and so is one of these $n_j$. Therefore, $m \leq n_{i-1+d(2-\log_2(\xi))}$, which contradicts the definition of $m$. 
\end{proof}

For use in Section \ref{sec:reasonable}, we record an analogue of Lemma \ref{lem:dense block} which holds under the `badly approximable' assumption of Theorems \ref{constant type} and \ref{thm:iet bad approx}.

\begin{lem}\label{lem:dense2}
If there exists some $\sigma>0$ such that $e_T(n)>\frac{\sigma}{n}$ for all $n$, then there exists some $K>0$ such that $\{T^i x\}_{i=1}^n$ is $\frac{K}{n}$-dense for all $x, n$.
\end{lem}

\begin{proof}
By choosing $\xi<\frac\sigma2$, we may choose $n_i=2^i$. By the previous lemma $\{x,...,T^nx\}$ is $2^{-[m-d(2-\log_2(\xi))]}$-dense where $m=\lfloor\log_2(n)\rfloor$. The lemma follows.
\end{proof}

We now prove Proposition \ref{prop:correlation}.

\begin{proof}[Proof of Proposition \ref{prop:correlation}]

Let $u=\max(\{j\} \cup \{i: n_i<\frac{1}{r_{n_j}}\})$. Let $v=d(2-\log_2(\xi))+\hat{q}$ where $\hat q$ is provided by Theorem \ref{thm:quant bc}. We divide up our sum as follows:
\begin{align}
	\sum_{i=j+1}^{\infty}\Big|\int g_ig_j-&\int g_j \int g_i \Big|  \nonumber \\
		& = \sum_{j < i \leq u }\Big|\int g_ig_j-\int g_j \int g_i \Big|+\sum_{u<i\leq u+4v} \Big|\int g_i g_j-\int g_i \int g_j \Big| \nonumber \\
		& +\sum_{i>u+4v} \Big|\int g_ig_j-\int g_j \int g_i \Big|.  \nonumber
\end{align}

\textit{Step 1:} We estimate the first term, in the case $j<u$ (otherwise there is no contribution from this term). By Lemma \ref{interval},  $T^{s+\ell}x \notin B(T^\ell x,r_\ell)$ for all $s$ so that $e_T(s+1)\geq 2r_\ell$. By definition of the $n_i$ and the choice of $u$, $e_T(2n_u) > \frac{\xi}{2n_u} > \frac{\xi}{2}r_{n_j}$. Then for any $\ell\geq n_j$, there are at most $1+4r_{n_j}/(\frac{\xi}{2}r_{n_j}) = 1+\frac{8}{\xi}$ disjoint intervals of size $\frac{\xi}{2n_u}$ intersecting $B(T^\ell x, 2r_\ell)$. It follows that if $\ell\geq n_{j}$ then 
\begin{equation} \label{eq:indep L inf bound}
	\left|\{s<2n_u:T^sx\in B(T^\ell x,2r_\ell)\}\right|\leq 1+\frac{8}{\xi}.
\end{equation} 
For any $s\in [n_i, 2n_i)$ and $\ell\in [n_j, 2n_j)$, $s>\ell$, $r_s<r_l$ and therefore $B(T^sx,r_s) \cap B(T^\ell x, r_\ell)\neq \emptyset$ only if $T^sx \in B(T^\ell x, 2r_\ell)$. Therefore,
\begin{equation}\label{eq:step1 1}
 \sum_{j< i \leq u} \int g_i g_j \leq \left(1+\frac{8}{\xi}\right) \Vert g_j\Vert_1 \leq \left(1+\frac{8}{\xi}\right) \Vert g_{j-1}\Vert_1
 \end{equation}
using Lemma \ref{lem:L1 bound}.

Now we bound $\sum_{j<i \leq u} \int g_i\int g_j$ above in terms of $\Vert g_{j-1}\Vert_1$.  Observe that as $n_{i+1}\geq 2n_i$ and under the assumption that $j<u$,  $u-j\leq  \log_2(\frac{1}{n_jr_{n_j}})$. We can straight-forwardly bound $\Vert g_j \Vert_1 \leq 2n_j r_{n_j}$. This implies that 
\[ \log_2\left(\frac{2}{\Vert g_j \Vert_1}\right) \geq \log_2\left(\frac{1}{n_j r_{n_j}}\right) \geq u-j. \]
Therefore, using Lemma \ref{lem:L1 bound},
\begin{align}
	 \sum_{j<i\leq u} \Vert g_j \Vert_1 \Vert g_i \Vert_1 &\leq \Vert g_j \Vert_1 \sum_{j<i\leq u} \Vert g_j \Vert_1 \nonumber \\
	 			&  \leq \Vert g_j \Vert_1 \log_2\left(\frac{2}{\Vert g_j \Vert_1} \right) \Vert g_j \Vert_1 \nonumber \\
				&  \leq 2 \Vert g_j \Vert_1 \leq 2 \Vert g_{j-1}\Vert_1. \label{eq:step1 2}
\end{align}

Combining \eqref{eq:step1 1} and \eqref{eq:step1 2}, by the triangle inequality we have 
$$ \sum_{j < i \leq u }\Big|\int g_ig_j-\int g_j \int g_i \Big|\leq (2+1+\frac 8 \xi)\Vert g_{j-1} \Vert.$$
\vspace{.25cm}

\textit{Step 2:} We estimate the second term. By Lemma \ref{lem:hit bound} we have that there exists $D$ independent of $i,j$ with $\|\sum_{u<i\leq u+4v} g_i(x)\|_\infty<D$. We can then apply the triangle and H\"older inequalities as follows:
\begin{align}
	\sum_{u<i\leq u+4v} \Big|\int g_i g_j-\int g_i \int g_j \Big| &\leq \sum_{u<i\leq u+4v} \Big|\int g_i g_j| + \Big| \int g_i \int g_j \Big| \nonumber \\
			& \leq \sum_{u<i\leq u+4v} \|g_i\|_\infty \| g_j\|_1 + \|g_i\|_\infty \|g_j\|_1 \nonumber \\
			& \leq 2 \|g_j\|_1 \sum_{u<i\leq u+4v} \|g_i\|_\infty < 2D\|g_j\|_1 \leq 2D\|g_{j-1}\|_1, \nonumber
\end{align}
as desired, again using Lemma \ref{lem:L1 bound} at the last step.
\vspace{.25cm}

\textit{Step 3:}
We estimate the third term. To do this we will use Lemma \ref{lem:correl by invar} to show that $g_i$ is nearly independent from $f_{i,j}$, a function that is close to $g_j$ and is constructed with the help of Lemma \ref{lem:dense block}. We will then show that $g_i$ and $g_j$ are nearly independent, as desired.

For fixed $i,j$ with $i>u+4v$, let $b = \frac{3u+i}{4}$. Note that as $i>4v+u$ in these sum terms, $b-v>u$. Let $S_{i,j}$ be the set of discontinuities of $T^{n_b}$; in the terminology of Lemma \ref{lem:dense block}, $S_{i,j} = S_{n_b}$. By Lemma \ref{lem:dense block}, $S_{i,j} = S_{n_{b-v+\hat q+d(2-\log_2\xi)}}$ is $\frac{1}{n_{b-v+\hat q}}$-dense, and hence $\frac{1}{n_{b-v}}$-dense. As $n_{k+1}\geq 2n_k$, we have $n_{b-v} \geq n_{u+1} 2^{b-v-u-1} =n_{u+1} 2^{\frac{i-u}{4}-v-1}$. Then
\[ \frac{1}{n_{b-v}} \leq \frac{1}{n_{u+1}}2^{-\frac{i-u}{4}+v+1}\]
for all $i>u+4v$.

Applying Lemma \ref{lem:const on blocks} to $g_j$ using the $\frac{1}{n_{b-v}}$-dense set $S_{i,j}$, we obtain a function $f_{i,j}$ which is constant on each element of the partition by $S_{i,j}$ and such that 
\begin{equation}\label{eq:indep close bound} 
	\Vert f_{i,j} - g_j \Vert_1 \leq 2n_j \frac{1}{n_{u+1}} 2^{-\frac{i-u}{4}+v+1} \ \ \mbox{ and } \ \  \Vert f_{i,j} \Vert_\infty \leq \Vert g_j \Vert_\infty.
\end{equation}
In addition, $f_{i,j}=\sum_{l=1}^R \alpha_l \chi_{J_l}$ where $J_l$ are disjoint intervals from the partition by $S_{i,j}$ and $\alpha_l \geq 0$.

We have the following lower bound on $\| g_{j-1}\|_1$: $\| g_{j-1}\|_1 \geq 2n_{j-1}r_{2n_{j-1}} \geq n_j r_{n_j}$ using the Khinchin condition. By our choice of $u$, $n_{u+1} \geq \frac{1}{r_{n_j}}$, so $\| g_{j-1}\|_1 \geq n_j \frac{1}{n_{u+1}}.$  From this and \eqref{eq:indep close bound} we obtain
\begin{equation}\label{eq:f close to g}
 	\Vert f_{i,j} - g_j \Vert_1 \leq K' 2^{-\frac{i-u}{4}} \Vert g_{j-1} \Vert_1
\end{equation}
for a constant $K'$ independent of $j$.

We apply Theorem \ref{thm:quant bc}, to an arbitrary $n_b$-block interval $J$ as follows. Let $L=\frac{i-u}{4}-\hat q$; then $n_{b+\hat q +L} = n_{\frac{i+u}{2}}$. Since $i>u+4v$ in this step, using the definition of $v$ we see that $L>0$ as necessary for Theorem \ref{thm:quant bc}. Then applying the theorem (with $b$ here in the role of Theorem \ref{thm:quant bc}'s $i$), we obtain for any $x, x'$ and any such $J$:
\[ \Big| \sum_{j=1}^{n_{\frac{i+u}{2}}} \chi_J(T^jx) - \chi_J(T^jx') \Big|<  n_{\frac{i+u}{2}} C_1 e^{-C_2(\frac{i-u}{4}-\hat{q})} |J|\leq  n_{\frac{i+u}{2}} C_1' e^{-C_2\frac{i-u}{4}} |J| \]
with constants $C_1$ and $C_2$ as in Theorem \ref{thm:quant bc} and $C_1'=C_1e^{C_2\hat q}$.

Since this bound holds for all $x'$, it holds when we replace the $x'$ term with its average over all $x'$. We obtain
\[ \Big| \big| \{ 0 < k \leq n_{\frac{i+u}{2}} : T^kx\in J \} \big|  - n_{\frac{i+u}{2}} |J| \Big| <  n_{\frac{i+u}{2}} C_1' e^{-C_2\frac{i-u}{4}}|J|\]
for positive constants $C_1'$ and $C_2$ which are independent of $j$.

Now consider
\begin{align}
	\left| \int g_i f_{i,j} - \int g_i \int f_{i,j} \right| &= \left| \sum_{l=1}^R \alpha_l \left( \int g_i \chi_{J_l} - |J_l| \int g_i  \right)\right| \nonumber \\
		& \leq \sum_{l=1}^R \alpha_l \left| \int g_i \chi_{J_l} - |J_l| \int g_i \right|. \nonumber
\end{align}

Applying Lemma \ref{lem:correl by invar} with $h = g_i$, $n=n_{\frac{i+u}{2}}$ and $\delta' = C_1' e^{-C_2\frac{i-u}4} |J_l|$, we obtain
\begin{align}
	\left|  \int g_i f_{i,j} - \int g_i \int f_{i,j} \right| &\leq \sum_{l=1}^R \alpha_l \left( C_1' e^{-C_2\frac{i-u}4} |J_l|\cdot \|g_i\|_1 + \frac{1}{n_{\frac{i+u}{2}}} \sum_{k=1}^{n_{\frac{i+u}{2}}} \int_{J_l} |g_i-g_i\circ T^k| dx \right) \nonumber \\
			& = \|f_{i,j}\|_1 C_1' e^{-C_2\frac{i-u}{4}}\Vert g_{j-1} \Vert_1 + \frac{1}{n_{\frac{i+u}{2}}} \sum_{k=1}^{n_{\frac{i+u}{2}}} \sum_{l=1}^R \alpha_l \int_{J_l} |g_i-g_i\circ T^k| dx \nonumber \\
			& \leq \|f_{i,j}\|_1 C_1' e^{-C_2\frac{i-u}{4}}\Vert g_{j-1} \Vert_1 + \| f_{i,j}\|_\infty \frac{1}{n_{\frac{i+u}{2}}} \sum_{k=1}^{n_{\frac{i+u}{2}}} \| g_i-g_i\circ T^k \|_1 \nonumber \\
			&\leq D' C_1' e^{-C_2\frac{i-u}{4}}\Vert g_{j-1} \Vert_1 + D' c_{i,j} \label{eqn:g f correl}
\end{align}
because $\| f_{i,j}\|_\infty \leq \|g_j\|_\infty\leq D'$ independent of $j$ by Lemma \ref{lem:hit bound}, and the $J_l$ are disjoint, and where
\[c_{i,j} = \max \left\{ \Vert g_i - g_i\circ T^k \Vert_1 : 0 \leq k \leq n_{\frac{i+u}{2}} \right\}. \] 
We have also used that Lemma \ref{lem:L1 bound} implies $\Vert g_i \Vert_1 \leq \Vert g_{j-1} \Vert_1$.

By Lemma \ref{lem:almost invar}, 
\[ \sum_{i>j} c_{i,j}  \leq D \Vert g_{j-1} \Vert_1 \]
for a constant $D$ independent of $j$. Combining this and equation \eqref{eqn:g f correl}, we get
\begin{equation}\label{eqn:sum g f}
	\sum_{i>u+4v} \left| \int g_i f_{i,j} - \int g_i \int f_{i,j} \right| < \hat D \Vert g_{j-1} \Vert_1
\end{equation}
for some $\hat D>0$ independent of $j$.

Then we have
\begin{align}
	\Big| \int g_i g_j &- \int g_i \int g_j \Big| \nonumber \\
			& \leq \Big| \int g_i g_j - \int g_i f_{i,j} \Big| + \Big| \int g_i f_{i,j} - \int g_i \int f_{i,j} \Big| + \Big| \int g_i \int f_{i,j} - \int g_i \int g_j \Big| \nonumber \\
			& \leq \Vert g_i\Vert_\infty \Vert g_j-f_{i,j} \Vert_1 + \Big| \int g_i f_{i,j} - \int g_i \int f_{i,j} \Big| + \Vert g_i \Vert_\infty \Vert f_{i,j} - g_j \Vert_1 \nonumber \\
			& \leq K'' 2^{-\frac{i-u}{4}} \Vert g_{j-1}\Vert_1 + \Big| \int g_i f_{i,j} - \int g_i \int f_{i,j} \Big| + K'' 2^{-\frac{i-u}{4}} \Vert g_{j-1} \Vert_1 \nonumber
\end{align}
for some $K''$ independent of $j$. The last inequality uses: Lemma \ref{lem:hit bound} to bound $\Vert g_i \Vert_\infty$ and \eqref{eq:f close to g} to bound $\Vert f_{i,j}-g_j\Vert_1$. Summing the above expression over the relevant $i$ and using equation \eqref{eqn:sum g f}, we get 
\[  \sum_{i>u+4v} \left| \int g_i g_j - \int g_i \int g_j \right| < D'' \Vert g_{j-1} \Vert_1 \]
for a constant $D''$ independent of $j$, as desired. This completes the proof.
\end{proof}

%

\subsection{Abstract setting: Proof of Proposition \ref{prop:abstract crit}}\label{sec:abstract}

We prove Proposition \ref{prop:abstract crit} below. First, we introduce some notation.

Let $H_i$ be as in Proposition \ref{prop:abstract crit}. Recall that these nonnegative random variables satisfy the following criteria for all $i$: 

\begin{itemize}
	\item[(H1)] $\Vert H_i \Vert_\infty < C_1$
	\item[(H2)] $\sum_{i=1}^\infty \int H_i = \infty$
	\item[(H3)] $\sum_{j=i+1}^\infty |\int H_iH_j - \int H_i \int H_j| < C_2 \Vert H_{i-1} \Vert_1$.
\end{itemize}
From (H1) it is immediate that $\Vert H_i \Vert_1<C_1$.

Let $F_i=H_i-\int H_i.$ Observe that $F_i$ satisfies the following for all $i$:
\begin{itemize}
	\item [(F1)] $\int F_i=0$
	\item [(F2)] $\Vert F_i\Vert_{\infty}<\Vert H_i \Vert_\infty < C_1$
	\item [(F3)] $\sum_{j=i+1}^{\infty}|\int F_iF_j|<C_2 \Vert H_{i-1}\Vert_1$
\end{itemize}
Again, it is easy to see that $\Vert F_i \Vert_1<2\Vert H_i \Vert_1$.

Let $m_0=0$ and define $m_k$ inductively by $m_{k+1}=\min \{i:\sum_{j=m_k+1}^{i}  \Vert H_j \Vert_1 \geq 1\}.$ Condition (H2) guarantees the existence of $m_k$ for all $k$. From this definition, and from the fact, noted above, that $\Vert H_i \Vert_1 <C_1$ for all $i$, we have that 
\begin{equation}\label{eqn:mk bounds}
	1 \leq \sum_{i=m_k+1}^{m_{k+1}} \Vert H_i \Vert_1 < C_1+1.
\end{equation}

For the proof of Proposition \ref{prop:abstract crit} we use the following two classical results:

\begin{lem} (Chebyshev's inequality) Let $R$ be a random variable with $\int R d\mu=0$ and finite variance. Then $\mu(\{\omega: R(\omega)>c\}) \leq \frac{\int R^2 d\mu}{c^2}.$
\end{lem}

\begin{lem}(Borel-Cantelli) If $A_1,...$ are $m$-measurable sets and $\sum_{i=1}^{\infty}m(A_i)<\infty$ then $m(\{x: x\in A_i \text{ for infinitely many }i\})=0.$
\end{lem}

We will prove that 
\begin{equation}\label{eqn:Fi growth} 
	\lim_{N\to \infty} \frac{\sum_{i=1}^N F_i(x)}{\sum_{i=1}^N \int H_i} = 0
\end{equation}
for a.e. $x$, which implies Proposition \ref{prop:abstract crit}. Indeed \eqref{eqn:Fi growth} implies that 
$$\underset{N \to \infty}{\lim}\frac{\sum_{i=1}^N H_i(x)}{\sum_{i=1}^N \int H_i}=\underset{N \to \infty}{\lim}
\frac{\sum_{i=1}^N F_i(x)}{\sum_{i=1}^N\int  H_i}+\frac{\sum_{i=1}^N \int H_i}{\sum_{i=1}^N \int H_i}=1.$$ 
Our proof is in two steps. First, we prove that \eqref{eqn:Fi growth} holds along the subsequence $\{ m_{N^2} \}_{N\in\mathbb{N}}$.

\begin{lem}\label{lem:conv on sq} 
\[ \underset{N \to \infty}{\lim} \frac{\sum_{i=1}^{m_{N^2}}F_i(x)}{N^2}=0 \]
for a.e. $x$.
\end{lem}

Note that by \eqref{eqn:mk bounds}, $\sum_{i=1}^{m_{N^2}} \Vert H_i \Vert_1 \geq N^2$, so Lemma \ref{lem:conv on sq} implies Proposition \ref{prop:abstract crit} for this subsequence.

\begin{proof} 

Consider, for any $M$, the mean-zero random variable $\sum_{i=1}^{m_M} F_i(x)$. We want to bound its second moment.
\[ \int\left(\sum_{i=1}^{m_M} F_i(x)\right)^2dx = \int\left( \sum_{i=1}^{m_M} F_i(x)^2 + 2 \sum_{1\leq i<j\leq m_M} F_i(x)F_j(x) \right)dx\]

First,
\begin{align}
	\sum_{i=1}^{m_M} \int F_i(x)^2 dx & \leq \sum_{i=1}^{m_M} \Vert F_i \Vert_\infty \Vert F_i \Vert_1 \nonumber \\
				& < C_1 \sum_{i=1}^{m_M} \Vert F_i\Vert_1 < 2C_1 \sum_{i=1}^{m_M} \Vert H_i \Vert_1 < 2C_1(C_1+1) M. \nonumber
\end{align}
using the H\"older inequality, our bounds on $\Vert F_i \Vert_*$, and equation \eqref{eqn:mk bounds}.

Second,
\begin{align}
	\Big| 2\sum_{1\leq i<j\leq m_M} \int F_i(x)F_j(x) dx \Big|& \leq  2 \sum_{i=1}^{m_M-1} \Big| \sum_{j=i+1}^{m_M} \int F_i(x)F_j(x) dx\Big| \nonumber \\
						& \leq 2\sum_{i=1}^{m_M} C_2 \Vert H_{i-1}\Vert_1 < 2C_2(C_1+1)M \nonumber
\end{align}
using property (F3) and equation \eqref{eqn:mk bounds}.

We conclude that $\int (\sum_{i=1}^{m_M} F_i(x))^2dx < \tilde C M$ for some positive constant $\tilde C$ and all $M$.

Now, by Chebyshev, for each $N$ and any $\delta>0$,
\[ m\Big( \big\{ x: \Big|\sum_{i=1}^{m_{N^2}} F_i(x) \Big| > \delta N^2\big\}\Big)  <  \frac{\tilde C N^2}{\delta^2 N^4} = \frac{\tilde C}{\delta^2 N^2}.\]
Let $A_N = \{ x: |\sum_{i=1}^{m_{N^2}} F_i(x) | > \delta N^2\}$. By the above, this sequence of sets has summable measure, so by Borel-Cantelli, for almost all $x$,
\[ \limsup_{N\to \infty} \frac{|\sum_{i=1}^{m_{N^2}} F_i(x) |}{N^2} =0\]
proving the lemma.
\end{proof}

We are now ready to prove Proposition \ref{prop:abstract crit}.

\begin{proof}[Proof of Proposition \ref{prop:abstract crit}] 

We want to show 
\[ \lim_{r\to \infty} \frac{\sum_{i=1}^r F_i(x) }{ \sum_{i=1}^r \int H_i}=0.\]
Choose $N$ so that $m_{N^2}\leq r < m_{(N+1)^2}$. Again using \eqref{eqn:mk bounds},
\[ \left| \frac{\sum_{i=1}^r F_i(x)}{\sum_{i=1}^r \int H_i} \right| \leq \left| \frac{\sum_{i=1}^{m_{N^2}} F_i(x)}{N^2} \right| + \left| \frac{\sum_{i=m_{N^2}+1}^r F_i(x)}{N^2} \right|.\]
Therefore, using Lemma \ref{lem:conv on sq}, it is sufficient to prove that for almost every $x$,
\[ \lim_{N\to \infty} \max_{m_{N^2}< r < m_{(N+1)^2}} \frac{\sum_{i=m_{N^2}+1}^r F_i(x)}{N^2} = 0.  \]
Recalling the definition of $F_i$, we need to consider
\[ \frac{\sum_{i=m_{N^2}+1}^r H_i(x) - \Vert H_i \Vert_1 }{N^2}.\]

The proof follows an argument similar to Lemma \ref{lem:conv on sq}. For any $L<m_{(N+1)^2}$, using the bounds on $\Vert F_i \Vert_*$ and equation \eqref{eqn:mk bounds}, one has $\int(\sum_{i=m_{N^2}+1}^L F_i(x))^2dx <\tilde C N$. Chebyshev's inequality implies
\[ m\Big( \big\{ x: \Big|\sum_{i=m_{N^2}+1}^L F_i(x) \Big| > \delta N^2\big\}\Big)  <  \frac{\tilde C N}{\delta^2 N^4}. \]
This is summable, so applying Borel-Cantelli as before, the set of $x$ which do not have the desired convergence property has measure zero.

\end{proof}

\begin{rem}\label{rem:h3 prime}
Note that condition (H3) can be replaced by the following slightly weaker condition, which is all that is used in the proof of Proposition \ref{prop:abstract crit}:

(H$3'$): There exists some constant $C_2$ such that for all $N$,
\[ \left|\sum_{i=1}^{N-1} \sum_{j=i+1}^N |\int H_i H_j - \int H_i \int H_j\right| < C_2 \sum_{i=1}^{N-1} \|H_i\|_1.\]
When we use Proposition \ref{prop:abstract crit} in the proof of Proposition \ref{prop:reasonable i}, we will use (H$3'$) in place of (H3).
\end{rem}

%

\subsection{Controlling the omitted terms}\label{sec:limsup}

We now turn our attention to $\sum_{j \notin \cup [n_i,2n_i)}\chi_{B(\frac 1 2,r_i)}(T^ix)$, that is, the terms omitted in our consideration of $g_i$.

Recall that $\beta_i(x)=\sum_{j=2n_i}^{n_{i+1}-1}\chi_{B(\frac 1 2,r_j)}(T^jx)$, where we understand that $\beta_i\equiv0$ if $n_{i+1}=2n_i$. Notice that it is possible that $\beta_i\equiv0$ for many $i$.  As we will see below, the assumptions on $T$ in Theorem \ref{thm:explicit} will imply that for most $i$, $\beta_i$ contributes little to the sum we are considering. This will enable us to prove the main result of this section:

\begin{prop}\label{prop:limsup}
Under the assumptions of Theorem \ref{thm:explicit}, for any $\epsilon>0$ there exists $\xi_0>0$ so that if $\xi_0>\xi>0$ then for almost every $x$ we have $\sum_{i=1}^{N-1}\beta_i(x)<\epsilon\sum_{i=1}^{n_N-1}2r_i$ for all sufficiently large $N$. ($N$ is allowed to depend on $x$.)
\end{prop}

The first step is to prove a version of Lemma \ref{lem:hit bound} for our current setting, a bound on $\Vert \beta_i\Vert_\infty$. We accomplish that in the following Lemma and Corollary. Recall that $a_{i+1}=\frac{n_{i+1}}{n_i}$.

\begin{lem}\label{lem:large a}
If $\|\beta_k\|_\infty>\max\{\frac{10}{\xi}, \frac{10}{\xi}(\sum_{i=1}^k \| g_i \|_1)^{\frac{2}{3}}\}$, then there exists a constant $C>0$ such that $a_{k+1}>Ck^{\frac{2}{3}}$. $C$ depends only on $r_1$, the first term in our Khinchin sequence.
\end{lem}

\begin{proof}
Because $\{T^ix\}_{i=0}^{n_{k+1}-1}$ is $e_{T}(n_{k+1})>\frac {\xi}{n_{k+1}}=\frac {\xi}{a_{i+1}n_i}$-separated, 
$\|\beta_k\|_\infty \leq \frac{2 a_{k+1}}{\xi}n_kr_{2n_k}+1$, using the argument of Lemma \ref{lem:hit bound}. As for almost every $x$, $\beta_k(x) >\frac{10}{\xi}$ and $\xi$ is small, we can reformulate this bound as $\|\beta_k\|_\infty \leq \frac{4 a_{k+1}}{\xi}n_kr_{2n_k}$.  Note, in addition, that by our Khinchin condition, $\| g_i\|_1 \geq 2n_kr_{2n_k}$ for all $i\leq k$ and so $(\sum_{i=1}^k \| g_i \|_1 )^{\frac{2}{3}}\geq (k2n_kr_{2n_k})^{\frac{2}{3}}.$ 

Using that $\|\beta_k\|_\infty \leq \frac{4 a_{k+1}}{\xi}n_kr_{2n_k}$, and by assumption, $\|\beta_k\|_\infty > \frac{10}{\xi}(\sum_{i=1}^k \|g_i\|_1)^{\frac{2}{3}} \geq \frac{10}{\xi} (k2n_kr_{2n_k})^{\frac{2}{3}}$, 
\[ \frac{4}{\xi} a_{k+1} n_kr_{2n_k} > \frac{10}{\xi} (k2n_kr_{2n_k})^{\frac{2}{3}}\]
and so
\[ a_{k+1}(n_kr_{2n_k})^{\frac{1}{3}} > k^{\frac{2}{3}}.\]
Using this inequality and the fact that the Khinchin condition implies that $n_kr_{2n_k}\leq \frac{1}{2}r_1$, 
\[ a_{k+1} >  \left( \frac{2}{r_1} \right)^{\frac{1}{3}}k^{\frac{2}{3}},\]
proving the Lemma.
\end{proof}

\begin{cor}\label{cor:beta infty bound}
For almost every $x$, for all but finitely many $k$, $\beta_{k}(x)< \max\{\frac {10}{\xi},\frac{10}{\xi}(\sum_{i=1}^k \| g_i \|_1)^{\frac{2}{3}}\}$.
\end{cor}

\begin{proof} 
First, by our assumptions that $a_k \leq k^{\frac{4}{3}}$ for all but finitely many $k$, we have that for such $k$, $\|\beta_k(x)\|_1\leq \sum_{j=1}^{\frac{4}{3} \log_2(k)}2^jn_kr_{2^jn_k}$. By the Khinchin condition on $\{r_i\}$ this is $O(\log(k)n_kr_{2n_k})$. 
Recall that in the proof of Lemma \ref{lem:large a}, we saw that $(\sum_{i=1}^k\| g_i \|)^{\frac{2}{3}}\geq (kn_kr_{2n_k})^{\frac{2}{3}}$.

We claim that for such $k$,
\[\lambda\left(\left\{x:\beta_k(x)>\max\{\frac {10}\xi,\frac{10}\xi(\sum_{i=1}^k\| g_i \|_1)^{\frac{2}{3}}\}\right\} \right) = O\left(\frac {\log(k)} {k^{\frac 2 3 }}\right).\]
Indeed, using the estimates noted in the previous paragraph,
\begin{multline*}
	\lambda\left(\{x:\beta(x)>\frac{10}\xi(\sum_{i=1}^k\| g_i \|)^{\frac{2}{3}}\}\right) \leq \frac{\|\beta_k\|_1}{\frac{10}{\xi}(kn_kr_{2n_k})^{\frac 2 3 }}\leq \\
O\left(\frac{\log(k)(n_kr_{2n_k})^{\frac{1}{3}}}{k^{\frac{2}{3}}}\right)=O\left(\frac{\log(k)}{k^{\frac{2}{3}}}\right).
\end{multline*}
The last step uses that $\{r_i\}$ is a Khinchine sequence and so $n_kr_{2n_k}<\frac{1}{2}r_1$.  

By Lemma \ref{lem:large a}, for large $k$, the set of $x$ for which $\beta_k(x)$ has such large values has positive measure only if $a_k\geq k^{\frac 1 2}$. But assumption (3) of Theorem 2.3 is that $\sum_{k:a_k>k^{\frac{1}{2}}}\frac {\log(k)} {k^{\frac{2}{3}}}<\infty$. This implies the corollary via the Borel-Cantelli Lemma.
\end{proof}

The next step is the following probabilistic result, which is an analogue of Proposition \ref{prop:abstract crit}:

\begin{lem}\label{lem:abstract bound} 
Let $K_i:[0,1) \to \mathbb{R}_{\geq 0}$ be a sequence of functions and $C_N$ an increasing, unbounded, positive sequence of real numbers $C_N = o(N^3)$ satisfying the following:

\begin{itemize}
	\item[(K1)] There exists some $M>0$ such that $\sum_{i=1}^N \Vert K_i \Vert_1<C_N$ for all $N>M$
	\item[(K2)] There exists $D_0>0$ such that $\max_{i<N, x} \{K_i(x)\} < D_0C_N^{\frac {2}{3}}$ 
	\item[(K3)] There exists $D_1>0$ such that $\sum_{1\leq i<j<N} \left(\int K_i(x) K_j(x) - \int K_i \int K_j\right) < D_1C_N^{\frac{5}{3}}$
\end{itemize}
Then for almost every $x$ 
\[\limsup_{N\to \infty} \frac{\sum_{i=1}^N K_i(x)-\sum_{i=1}^N \|K_i\|_1}{C_N}=0.\]
\end{lem}

\begin{proof}
Let $R_i=K_i-\int K_i$. Note that $\Vert R_i \Vert _1 \leq 2 \Vert K_i \Vert_1$ so $\sum_{i=1}^N \Vert R_i\Vert_1 < 2C_N$ for $N>M$, that (K2) implies $\Vert R_i \Vert_\infty < D_0 C_N^{\frac{2}{3}}$, and that (K3) implies $\sum_{1\leq i<j\leq N} \int R_i R_j < D_1 C_N^{\frac{5}{3}}.$

We begin by computing the variance of $\sum_{i=1}^{N-1} R_i$. Because $\|R_i\|_2^2\leq \|R_i\|_1\cdot \|R_i\|_\infty$ we obtain $\sum_{i=1}^{N-1} \Vert R_i \Vert_2^2 \leq 2D_0 C_N^{\frac{5}{3}}$ using (K1) and (K2). Using (K3), $ 2\sum_{1\leq i < j < N} \int R_i R_j  \leq 2 D_1 C_N^{\frac{5}{3}}.$ Therefore $\int (\sum_{i=1}^{N-1} R_i(x))^2 dx\leq  2(D_0 + D_1) C_N^{\frac{5}{3}}.$

Fix any $\delta>0$. By Chebyshev's inequality,
\begin{equation} \label{eq:cheby beta}
	m\left(\{x: \sum_{i=1}^{N-1} R_i(x)>\delta C_N\}\right)\leq \frac{2(D_0+D_1)}{\delta^2C_N^{1/3}}.
\end{equation}

Recall that the $C_N$ are increasing and without bound. For any $r$, let ${k_r=\min\{N:C_N>r\}}$. Note that $C_{k_r}>r$ by definition. In addition, since $(N+1)^4 - N^4$ is $O(N^3)$ and $C_N = o(N^3)$, for sufficiently large $N$, $C_{k_{N^4}}<(N+1)^4$.

Consider $\{ x : \sum_{i=1}^{k_{N^4}-1} R_i(x) > \delta C_{k_{N^4}} \}.$ By \eqref{eq:cheby beta}
\[ m\left(\{ x : \sum_{i=1}^{k_{N^4}-1} R_i(x) > \delta C_{k_{N^4}} \}\right) < \frac{2(D_0+D_1)}{\delta^2 C^{\frac{1}{3}}_{k_{N^4}} } < \frac{2(D_0+D_1)}{\delta^2 N^{\frac{4}{3}}}\]
since $C_{k_{N^4}} > N^4$. These measures form a summable series, so by the Borel-Cantelli Lemma, for almost all $x$, $\sum_{i=1}^{k_{N^4}-1} R_i(x) > \delta C_{k_{N^4}} >\delta N^4$ for only finitely many $N$. Therefore, for almost all $x$,
\[ \limsup_{N\to \infty} \frac{|\sum_{i=1}^{k_{N^{4}}}R_i(x)|}{N^{4}}\leq \delta. \]
This establishes the desired convergence along the sequence $\{k_{N^4}-1\}$. 

We now need to consider the omitted terms. Consider
\begin{equation}\label{eqn:max}
	\max_{k_{N^4} \leq L < k_{(N+1)^4}} \sum_{i=k_{N^4}}^{L} R_i(x) \leq \sum_{i=k_{N^4}}^{k_{(N+1)^4}} K_i(x).
\end{equation}
(The inequality holds as $R_i+\Vert K_i \Vert_1 = K_i \geq 0$.) Again, we bound the variance, using (K1), (K2), and (K3). (K1) and (K2) imply $\sum_{i=k_{N^4}}^{k_{(N+1)^4}} \Vert K_i \Vert_2^2 \leq D_0 C_{k_{(N+1)^4}}^{\frac{5}{3}}$. With (K3), we get an upper bound on the variance of \eqref{eqn:max} of 
\[(D_0+2D_1) C_{k_{(N+1)^4}}^{\frac{5}{3}} < (D_0+2D_1) ((N+2)^4)^{\frac{5}{3}} = (D_0+2D_1) (N+2)^{\frac{20}{3}}\]
for all $N$ sufficiently large. At the last step we have used the fact that for sufficiently large $N$, $C_{k_{(N+1)^4}}<(N+2)^4$, which relies on the $C_N=o(N^3)$ assumption.

By Chebyshev's inequality 
\begin{align}
	m\left(\left\{x:\underset{L<k_{(N+1)^4}}{\max} |\sum_{i=k_{N^4}}^L R_i(x)|>\delta C_{k_{N^4}}\right\}\right)&\leq \frac{(D_0+2D_1)(N+2)^{20/3}}{(\delta C_{k_{N^4}})^2} \nonumber \\
				& \leq \frac{(D_0+2D_1) (N+2)^{20/3}}{\delta^2N^8} \nonumber \\
				& \leq 2 (D_0+2D_1) N^{-4/3} \delta^{-2} \nonumber
\end{align}
for sufficiently large $N$.

Therefore, by the Borel-Cantelli Lemma almost every $x$ has ${|\sum_{i=k_{N^4}}^L R_i(x)|>\delta N^{4}}$ with $L<k_{(N+1)^4}$ only finitely many times. Therefore, for any integer $N$, writing $N = k_{m^4}+L$ with $m$ the largest integer such that $k_{m^4}\leq N$, we get
\[\limsup_{N\to \infty}\frac{|\sum_{i=1}^{N}R_i(x)|}{C_N}\leq 2\delta\]
for almost every $x$. Letting $\delta\to 0$ finishes the proof.
\end{proof}

\begin{lem}\label{lem:translation} 
Under the assumptions of Theorem \ref{thm:explicit}, for any $\epsilon>0$ there exists $\xi_0>0$ so that if $\xi_0>\xi>0$ then 
	\[ \limsup_{N \to \infty} \frac{\sum_{i=1}^{N-1} \Vert \beta_i \Vert_1}{\sum_{i=1}^{n_N-1} 2r_i}<\epsilon\]
where the $\beta_i(x)=\beta_i(x,\xi)$ are calculated using $n_i(\xi)$ and $a_i(\xi)$.
\end{lem}

\begin{proof}
Let $\epsilon>0$ be given. Fix some $\xi_1$ for which the assumptions of Theorem \ref{thm:explicit} hold. Let $n_i$ and $\xi_i$ denote $n_i(\xi_1)$ and $a_i(\xi_1)$.

Notice that $e_T(2n_{i+1}) = e_T(2 a_{i+1} n_i) > \frac{\xi_1}{2 a_{i+1} n_i}$. Therefore, if $a_{i+1}<A$, then for any $j \leq a_{i+1}$, 
\[ e_T(2 j n_i) \geq e_T(2 a_{i+1} n_i) > \frac{\xi_0}{ j n_i} \]
where $\xi_0=\frac{\xi_1}{2A}$. For any choice of $\xi<\xi_0$, let  $n_i'=n_i(\xi)$ and $a_i'=a_i(\xi)$. By the choice of $\xi_0$, whenever $n_i'$ belongs to some $[n_l, n_{l+1})$ where $a_{l+1}<A$, we have that $a_{i+1}'=2$.

We will choose $A$ below. Once we have done so, fix $\xi$ less than $\xi_0=\frac{\xi_1}{2 A}$ and let $u_k=\max\{j:n'_j<n_k\}$. Then
\[\sum_{j=1}^{u_k}\sum_{i=n'_j}^{2n'_j}r_i\geq \sum_{i:i \in [n_l,n_{l+1}) \text{ and }a_{l+1}<A}^{n_k-1}r_i.\] 
Indeed, by the end of the previous paragraph, for any $i \in [n_j,n_{j+1})$ with $a_{j+1}<A$, $i \in [n'_j,2n'_j)$ for some $j$.

By condition (2) of Theorem \ref{thm:explicit}, if $A$ is sufficiently large, the upper density of $\{j: \exists i \in [2^j,2^{j+1}] \text{ with }i \notin \cup_{\ell=1}^{\infty}[n'_\ell,2n'_\ell)\}$ is less than $\epsilon$.  By the proof of Lemma \ref{lem:L1 bound}, whenever $i>j$ we have $\sum_{k=2^i}^{2^{i+1}-1} r_k\leq  \sum_{k=2^j}^{2^{j+1}-1}r_k$.

\begin{slem} 
If $\{s_i\}$ is a sequence of positive real numbers so that $s_i \leq s_j$ for all $i>j$ and $\sum s_i=\infty$, and if $\epsilon>0$, then for any $U \subset \mathbb{N}$ with upper density less than $\epsilon$ we have 
\[\limsup_{N \to \infty}\,\frac{\sum_{i \in U_N}s_i}{\sum_{i=1}^{N}s_i}\leq 2\epsilon\]
where $U_N=U\cap[1,N]$.
\end{slem}

\begin{proof}
There exists $M$ so that $|U_N|<\epsilon N$ for all $N>M$. Given such an $M$, write $U\cap (M, \infty)$ as $i_1<i_2<\cdots$ and inductively assign to each $i_k\in U\cap (M,\infty)$ the set of indices $G_{i_k} = [(k-1)\lceil \frac{1}{2\epsilon}\rceil+1, k\lceil \frac{1}{2\epsilon}\rceil)$. Note that by our choice of $M$, each $G_i\subseteq [1,i)$. Then $\sum_{l\in G_i} s_l \geq \frac{1}{2\epsilon} s_i$ by our assumption on $(s_i)$. Therefore
\begin{equation}\label{eqn:s_i bound}
	\sum_{l=1}^N s_l \geq \sum_{i\in U \cap (M,N]} \sum_{l\in G_i} s_l \geq \sum_{i\in U\cap[M,N]} \frac{1}{2\epsilon} s_i.
\end{equation}
Since $\sum s_i=\infty$, Clearly $\underset{N \to \infty}{\lim}\,\frac{\sum_{i\in U_M}s_i}{\sum_{i=1}^Ns_i}=0$. Therefore, \eqref{eqn:s_i bound} proves the desired result.
\end{proof}

For the $\beta_i'(x):=\beta_i(x,\xi)$ 
we have the bound 
\[ \sum_{i=1}^{N-1} \Vert \beta'_i \Vert_1 \leq \sum_{j\in U_{N-1}} \sum_{i=2^j}^{2^{j+1}-1}2r_i\]
where $U= \{j: \exists i \in [2^j,2^{j+1}] \text{ with }i \notin \cup_{\ell=1}^{\infty}[n'_\ell,2n'_\ell)\}$. Applying the Sublemma with $s_i=\sum_{k=2^i}^{2^{i+1}-1} 2r_k$ completes the proof of the lemma.
\end{proof}

We are now ready to prove Proposition \ref{prop:limsup}.

\begin{proof}[Proof of Proposition \ref{prop:limsup}] 

We prove the Proposition using Lemma \ref{lem:abstract bound}, with $K_i=\beta_i$ and $C_N=\sum_{i=1}^{n_N-1}2r_i$. To apply it, let $\tilde{\beta}_k(x)=\min\{\beta_k(x),C_k^{\frac 2 3}\}.$ By Corollary \ref{cor:beta infty bound} we have that for almost every $x$, $\beta_k(x)=\tilde{\beta}_k(x)$ for all but finitely many $k$ and so it is enough to prove Proposition \ref{prop:limsup} with $\beta_k$ replaced by $\tilde{\beta}_k$.
Lemma \ref{lem:abstract bound} controls the difference between $\sum_{i=1}^N\tilde{\beta}_i(x)$ and $\sum_{i=1}^N\|\tilde{\beta_i}\|_1$ relative to $C_N$. By Lemma \ref{lem:translation} (and the fact that $\tilde{\beta}_k(x)\leq\beta_k(x)$),  for any $\epsilon>0$ there exists $\xi_0$ so that for all $\xi_0>\xi>0$ we have the following control on $\sum_{i=1}^N\|\tilde{\beta_i}\|_1$:
\[\limsup_{N \to \infty}\frac{\sum_{i=1}^{N-1} \Vert \tilde{\beta_i} \Vert_1}{\sum_{i=1}^{n_N-1}2r_i}<\epsilon,\] 
and so by Lemma \ref{lem:abstract bound}, $\underset{N \to \infty}{\limsup}\, \frac{\sum_{i=1}^{N-1}\tilde{\beta}_i(x)}{\sum_{i=1}^{n_N-1}2r_i}<\epsilon$.
Therefore, to prove Proposition \ref{prop:limsup}, it suffices to check the conditions of Lemma \ref{lem:abstract bound} with $K_i=\tilde{\beta}_i$ and $C_N=\sum_{i=1}^{n_N-1}2r_i$. 

That $C_N$ is an increasing, unbounded, positive sequence is clear from its definition. The assumption that $a_i\leq i^{\frac{4}{3}}$ for all but finitely many $i$ implies that $n_N = O((N!)^\frac{4}{3})$. The Khinchin condition implies that $r_i\leq \frac{r_1}{i}$ for all $i$, so $C_N= O(\log n_N) = O(\log(N!)) = O(\sum_{i=1}^N \log i) = o(\sum_{i=1}^N i) = o(N^2)$ and so $C_N$ is certainly $o(N^3)$.

Condition (K1) follows from Lemma \ref{lem:translation}.

Condition (K2) follows immediately from the definition of $\tilde{\beta}_i$.

The proof of condition (K3) follows the argument of Proposition \ref{prop:correlation}; we sketch the argument here, using similar notation. Let $u=\max(\{j\} \cup \{i: n_i<\frac{1}{r_{n_{j+1}}} \} )$ and $v=d(2-\log_2(\xi))+\hat q$. (Note the slight difference in the definition of $u$.)

\

\noindent \emph{Step 1:} We bound the sum over indices $i$ satisfying $j<i<u$. Following the argument of Proposition \ref{prop:correlation} and replacing $n_u<\frac{1}{r_{n_j}}$ with $n_u<\frac{1}{r_{n_{j+1}}}$, we get $\sum_{j<i<u} \int \tilde{\beta_i}\tilde{\beta_j} \leq \left( 1+ \frac{8}{\xi}\right) \Vert\beta_j\Vert_1$. 

We bound $\sum_{j<i<u} \int \tilde{\beta_i} \int \tilde{\beta_j}$ as follows. Note that $\beta_{j+1}, \ldots \beta_{u-1}$ are sums whose terms have indices between $2n_{j+1}$ and $n_u$. This range of indices can be partitioned into $\log_2\left( \frac{n_u}{2n_{j+1}}\right)$ intervals between successive powers of 2. Then, using Lemma \ref{lem:L1 bound} to bound the contribution of each portion of this sum between successive powers of 2 by $\Vert g_{j+1} \Vert_1$, we get
\[\sum_{j<i<u}\Vert \tilde{\beta}_i\Vert_1 \Vert \tilde{\beta}_j\Vert_1\leq  \sum_{j<i<u}\Vert \beta_i\Vert_1 \Vert \beta_j\Vert_1 \leq \Vert \beta_j\Vert_1 \log_2\left( \frac{n_u}{2n_{j+1}} \right) \Vert g_{j+1}\Vert_1. \]
Note that $\Vert g_{j+1}\Vert_1 \leq 2n_{j+1}r_{n_{j+1}}$ and that, using the definition of $u$, $\frac{n_u}{n_{j+1}} < \frac{1}{n_{j+1}r_{n_{j+1}}}$. Therefore,
\[ \sum_{j<i<u}\Vert \tilde\beta_i\Vert_1 \Vert \tilde\beta_j\Vert_1 \leq \Vert \beta_j\Vert_1 \log_2\left( \frac{1}{2n_{j+1}r_{n_{j+1}}} \right)  (2n_{j+1}r_{n_{j+1}})\leq \Vert \beta_j \Vert_1.\]

Altogether, summing over $j$ as well,
\[ \sum_{1\leq j < N} \sum_{j<i<u}^{N-1} \left| \int \tilde{\beta}_i\tilde{\beta}_j - \int\tilde{\beta}_i\int\tilde{\beta}_j \right| \leq \sum_{1\leq j <N} C \Vert \beta_j\Vert_1 \leq C C_N\]
for some constant $C>0$, which is a sufficient bound for this part of the double sum.

\

\noindent \emph{Step 2:} For some constant $K'$ independent of $j$ we can bound
\[\sum_{j+1\leq i \leq j+4v}^{N-1} \left|\int \tilde{\beta}_i \tilde{\beta}_j - \int \tilde{\beta}_i\int\tilde{\beta}_j\right|  \leq \sum_{j+1 \leq i \leq j+4v}^{N-1} 2\Vert \beta_i\Vert_\infty \Vert \beta_j\Vert_1  \leq K' C_N^{\frac{2}{3}} \Vert \beta_j\Vert_1 \]
using (K2) to bound $\Vert\beta_i \Vert_\infty$. Summing over all $1\leq j< N$ gives a bound of $K'C_N^{\frac{5}{3}}$, as desired.

\

\noindent \emph{Step 3:} For the terms with indices $u+4v < i < N$, let $b=\frac{3u+i}{4}$. We approximate $\tilde{\beta}_j$ by a function $f_{i,j}$, constant on the elements of the partition by $S_{i,j}$. We find that
\begin{equation}\label{eqn:exp bound}
	\Vert f_{i,j} - \tilde{\beta}_j \Vert_1 \leq 2a_{j+1}n_j \frac{1}{n_{b-v}} \leq \tilde K 2^{-\frac{i-u}{4}} \Vert g_j \Vert_1 \  \ \mbox{ and } \ \ \Vert f_{i,j}\Vert_\infty \leq \Vert \tilde{\beta}_j \Vert_\infty.
\end{equation}
We use here that as $\Vert g_j \Vert_1 \geq 2n_j r_{2n_j} \geq  n_{j+1}r_{n_{j+1}}$ and $n_{u+1} \geq \frac{1}{r_{n_{j+1}}}$ (by our definition of $u$), $\Vert g_j\Vert_1 \geq \frac{n_{j+1}}{n_{u+1}}\geq\frac{a_{j+1}n_j}{n_{b-v}}$.

Applying Theorem \ref{thm:quant bc} exactly as in Proposition \ref{prop:correlation} gives that for any $n_b$-block interval,
\[ \Big| \sum_{j=1}^{n_{\frac{i+u}{2}}} \chi_J(T^jx) - \chi_J(T^jx') \Big| < n_{\frac{i+u}{2}} C_1 e^{-C_2\frac{i-u}{4}} |J| \]
with $C_1, C_2$ independent of $j$. Hence, as before,
\[ \Big| \big| \{ 0 < k \leq n_{\frac{i+u}{2}} : T^kx\in J \} \big|  - n_{\frac{i+u}{2}} |J| \Big| <  n_{\frac{i+u}{2}} C_1 e^{-C_2\frac{i-u}{4}}|J|.\]
Writing $f_{i,j} = \sum_{l=1}^R \alpha_l \chi_{J_l}(x)$, we consider
\[ \left| \int \tilde{\beta}_i f_{i,j} - \int \tilde{\beta}_i \int f_{i,j} \right| \leq \sum_{l=1}^R \alpha_l \left| \int \tilde{\beta}_i \chi_{J_l} - |J_l| \int \tilde{\beta}_i \right| \]
as in Proposition \ref{prop:correlation}. Again applying Lemma \ref{lem:correl by invar} with $h = \tilde\beta_i$, $n=n_{\frac{i+u}{2}}$ and $\delta' = C_1 e^{-C_2\frac{i-u}4} |J_l|$, we obtain
\begin{align}
	\left|  \int \tilde{\beta}_i f_{i,j} - \int \tilde{\beta}_i \int f_{i,j} \right| &\leq 
	\sum_{l=1}^R \alpha_l \left( C_1 e^{-C_2\frac{i-u}4} |J_l| \cdot \|\tilde{\beta}_i\|_1 + \frac{1}{n_{\frac{i+u}{2}}} \sum_{k=1}^{n_{\frac{i+u}{2}}} \int_{J_l} |\tilde{\beta}_i-\tilde{\beta}_i\circ T^k| dx \right) \nonumber \\
			& = \|f_{i,j}\|_1 C_1 e^{-C_2\frac{i-u}{4}}\Vert \beta_i \Vert_1 + \frac{1}{n_{\frac{i+u}{2}}} \sum_{k=1}^{n_{\frac{i+u}{2}}} \sum_{l=1}^R \alpha_l \int_{J_l} |\tilde{\beta}_i-\tilde{\beta}_i\circ T^k| dx \nonumber \\
			& \leq \|f_{i,j}\|_1 C_1 e^{-C_2\frac{i-u}{4}}\Vert \beta_i \Vert_1 + \| f_{i,j}\|_\infty \frac{1}{n_{\frac{i+u}{2}}} \sum_{k=1}^{n_{\frac{i+u}{2}}} \| \tilde{\beta}_i-\tilde{\beta}_i\circ T^k \|_1 \nonumber \\
			&\leq \|f_{i,j}\|_1 C_1 e^{-C_2\frac{i-u}{4}} \|\beta_i\|_1 + \| \tilde{\beta}_j \|_\infty \tilde c_{i,j} \label{eqn:beta int bound}
\end{align}
using the second statement of \eqref{eqn:exp bound}, where
\[ \tilde c_{i,j} = \max \{ \Vert \tilde{\beta}_i - \tilde{\beta}_i \circ T^s \Vert_1 : 0 \leq k \leq n_{\frac{i+j}{2}}\}. \]
Using Lemma \ref{lem:L1 bound}, we can bound $\|\beta_j\|_1$ and $\Vert \beta_i \Vert_1$ (and thus $\|\tilde{\beta}_j\|_1$ and $\Vert \tilde{\beta}_i \Vert_1$)  by $\log(a_{j+1})\|g_j\|_1$ and $\log (a_{i+1})\Vert g_j \Vert_1$, respectively. Because 
$$|\tilde{\beta}_i(x)-\tilde{\beta}_i(y)\|\leq |\beta_i(x)-\beta_i(y)|$$ for all $x,y$,  Lemma \ref{lem:beta almost invar}, $\sum_{i>u+4v} \tilde c_{i,j} \leq D \Vert g_j \Vert_1$ for some $D$ independent of $j$. Then summing \eqref{eqn:beta int bound} over the relevant indices and using (K2) to bound $\|\tilde{\beta}_j\|_\infty$ gives
\begin{multline*}
	\sum_{i>u+4v}^{N-1} \left| \int \tilde{\beta}_i f_{i,j} -  \int \tilde{\beta}_i \int f_{i,j}\right| \leq \\ 
	\left( \sum_{i>u+4v}^{N-1}C'e^{-C_2\frac{i-u}{4}} (\log(a_{i+1})\log(a_{j+1})) \right) \Vert g_j \Vert_1^2 + D' C_N^{\frac{2}{3}} \Vert g_j \Vert_1.
\end{multline*}
$C'$, and $C', C_2$, and $D'$ are independent of $j$.

A computation using $a_i\leq i^{4/3}$ for all but finitely many $i$ shows that $$\sum_{i>u+4v}^{N-1} C'e^{-C_2\frac{i-u}{4}} (\log(a_{i+1})\log(a_{j+1}))\|g_j\|_1^2 \leq L' C_N^{\frac 1 3}$$ for some $L'>0$ independent of $j$. Indeed, since there exists $C''$ so that $C'e^{-C_2\frac{i-u}{4}} (\log(i))<1$ for all $i>C''\log(\log(j))+j$,
$\sum_{i>u+4v}^{N-1} C'e^{-C_2\frac{i-u}{4}} (\log(a_{i+1})\leq C' \log(j)\log(\log(j))$. 
Considering separately the cases that $\|g_j\|_1< \frac 1 {\log(j)^2}$ and $\|g_j\|_1\geq \frac 1 {\log(j)^2}$ (which implies that $C_j\geq \frac{j}{\log(j)^2}$) we have the claim.

 We have
\begin{equation}\label{eqn:main bound}
	\sum_{i>u+4v}^{N-1} \left| \int \tilde{\beta}_i f_{i,j} -  \int \tilde{\beta}_i \int f_{i,j}\right| \leq L'C_N^{\frac 1 3} + D'C_N^{\frac{2}{3}} \Vert g_j \Vert_1. 
\end{equation}
Summing this over all $1\leq j<N$, we get a bound of $LC_N^{\frac{5}{3}}$ for some $L>0$ independent of $j$, as desired.

From this point, the proof follows the proof of Proposition \ref{prop:correlation}, combining estimates \eqref{eqn:exp bound} and \eqref{eqn:main bound} with the bounds from Steps 1 and 2 exactly as before. This completes (K3).
\end{proof}

We are now ready to complete the proof of Theorems \ref{khinchin seq} and \ref{thm:explicit}.

\begin{proof}[Proof of Theorems \ref{khinchin seq} and \ref{thm:explicit}]

We want to show that, under our conditions on $T$ and for almost every $x$,
\begin{equation}\label{eqn:goal}
	\lim_{M\to \infty} \frac{\sum_{j=1}^M \chi_{B(\frac{1}{2}, r_j)} (T^jx)}{\sum_{j=1}^M 2r_j}=1.
\end{equation}

Applying Proposition \ref{prop:abstract crit} with $H_i = g_i$ we have for almost all $x$ that
\begin{equation}\label{eqn:g eqn} 
	\lim_{N \to \infty} \frac{\sum_{i=1}^N g_i(x)}{\sum_{i=1}^N  \Vert g_i \Vert_1}=1.
\end{equation}
Then, when $M=2n_N-1$, we can decompose the numerator in equation \eqref{eqn:goal} as follows:
\begin{equation}\label{eqn:decomp}
	\frac{\sum_{j=1}^{2n_N-1} \chi_{B(\frac{1}{2}, r_j)} (T^jx)}{\sum_{j=1}^{2n_N-1} 2r_j} = \frac{\sum_{i=1}^{N} g_i(x) + \sum_{i=1}^{N-1} \beta_i(x)}{\sum_{j=1}^{2n_N-1} 2r_j}.
\end{equation}
Proposition \ref{prop:limsup} tells us that for almost every $x$,
\[ \limsup_{N\to \infty} \frac{\sum_{i=1}^{N-1} \beta_i(x)}{\sum_{j=1}^{2n_N-1} 2r_j} \leq \epsilon\]
so the contribution of the $\beta_i$ terms to equation \eqref{eqn:decomp} is negligible for large $N$, and they can be ignored:
\begin{equation}\label{eqn:new goal}
	\left|\lim_{N\to\infty} \frac{\sum_{j=1}^{2n_N-1} \chi_{B(\frac{1}{2}, r_j)} (T^jx)}{\sum_{j=1}^{2n_N-1} 2r_j} - \lim_{N\to\infty} \frac{\sum_{i=1}^{N} g_i(x) }{\sum_{j=1}^{2n_N-1} 2r_j}\right|<\epsilon.
\end{equation}

Note that for all $N$, 
\begin{equation}\label{eqn:gi rate}
	\frac{\sum_{i=1}^N \Vert g_i \Vert_1}{\sum_{j=1}^{2n_N-1} 2r_j} \leq 1.
\end{equation}
Combining equations \eqref{eqn:g eqn}, \eqref{eqn:gi rate} and \eqref{eqn:new goal} gives
\[ \limsup_{N\to\infty} \frac{\sum_{j=1}^{2n_N-1}\chi_{B(\frac{1}{2}, r_j)} (T^jx)}{\sum_{j=1}^{2n_N-1} 2r_j} \leq 1.\]

On the other hand, by Lemma \ref{lem:translation} using our second condition on $T$, for any $\delta>0$ there exists some $\xi>0$ so that, with $g_i$ defined using this $\xi$, we have
\begin{equation}\label{eqn:gi liminf}
	\liminf_{N\to\infty} \frac{\sum_{i=1}^N \Vert g_i \Vert_1}{\sum_{j=1}^{2n_{N}-1}2r_j} \geq 1-\delta.
\end{equation}
Using equations \eqref{eqn:gi liminf}, \eqref{eqn:gi rate} and \eqref{eqn:new goal} gives
\[ \liminf_{N\to\infty} \frac{\sum_{j=1}^{2n_N-1}\chi_{B(\frac{1}{2}, r_j)} (T^jx)}{\sum_{j=1}^{2n_N-1} 2r_j} \geq 1-\delta.\]

Letting $\delta \to 0$, we have now established equation \eqref{eqn:goal} along the sequence of times $\{2n_N-1\}$. This is sufficient. By (K2) the contribution of any terms with index in $[2n_N, n_{N+1})$ will be negligible for large $N$ and all $x$. The bound on $g_i(x)$ in Lemma \ref{lem:hit bound} tells us that for large $N$, the contribution of terms with index in $[n_{N+1}, 2n_{N+1}-1)$ will also be negligible. This completes the proof.

\end{proof}

%

\section{Proof of Theorem \ref{constant type}}

We now turn to the proof of Theorem \ref{constant type}. Recall that in this theorem the assumption that $\alpha$ is badly approximable allows us to omit the Khinchin condition and consider a wider class of radius sequences $\{r_i\}$.  As in Section \ref{sec:Thm1}, we will state and prove a generalization of Theorem \ref{constant type} to the case of interval exchange transformations. Using the notation developed in Section \ref{sec:setup}, this generalization is:

\begin{thm}\label{thm:iet bad approx}
Let $T$ be an IET satisfying the Keane condition so that there exists $\sigma>0$ with $e_T(n)>\frac{\sigma}n$ for all $n$.  Then for any decreasing sequence $\{r_i\}$ with divergent sum we have:
\[\lim_{N\to \infty}\frac{\sum_{j=1}^N \chi_{B(\frac{1}{2},r_j)}(T^j x)}{\sum_{j=1}^N 2r_j}=1\]
for almost every $x$.
\end{thm}

Let $\sigma$ be such that $e_T(n) > \frac{\sigma}n$ for all $n$.  If $T$ satisfies this for some $\sigma$, we say it is \emph{of constant type}.  Without loss of generality, we may assume $\sigma<1$.

For this section we adjust our definition of the $g_i$. For some constant $C>1$ (which we will choose later) let $g_i(x)=\sum_{j=C^{i}}^{C^{i+1}-1}\chi_{B(\frac{1}{2} ,r_j)}(T^jx)$.

The proof we provide is complicated by the fact that without the Khinchin condition on $r_i$ it is possible for $\Vert g_j \Vert_1\gg \Vert g_i \Vert_1$ for some $j>i$ (in contrast to Lemma \ref{lem:L1 bound}). This difficulty is handled for most values of $i$ by appealing directly to Theorem \ref{thm:quant bc}. We must then show that the remaining indices, which are not handled by our appeal to Theorem \ref{thm:quant bc}, make negligible contributions.

The outline of this section is as follows. We break up our indices into two disjoint sets according to a (fixed, large) parameter $M$. Section \ref{sec:reasonable} deals with those times $i$ such that $ir_i<M$. The proof in this section is similar to that in Section \ref{sec:Thm1} but simpler because we do not need to worry about the issues of Section \ref{sec:limsup}. Then in Section \ref{sec:big section} we treat the times $i$ such that $ir_i\geq M$. We partition them into a subset where we may apply Theorem \ref{thm:quant bc} and its complement, whose contributions we show are negligible. Lemma \ref{lem:partition} accomplishes the partitioning, Lemma \ref{lem:G bound} applies Theorem \ref{thm:quant bc}, and Corollary \ref{cor:B bound} controls the size of the blocks where we can not apply Theorem \ref{thm:quant bc}.  We note that the arguments in Section \ref{sec:reasonable} work for any value of $M$. It is for the proofs in Section \ref{sec:big section} that we have to choose a sufficiently large value of $M$.

Throughout this section, in an abuse of notation, $r_{C^L}$ denotes $r_{\lfloor C^L\rfloor}.$

%

\subsection{\boldmath{$ir_i$} small}\label{sec:reasonable}

In this subsection we treat $ir_i < M$.

\begin{prop}\label{prop:reasonable i}
Let $C, M$ be given. Let $E=\{i:r_{C^{i}}<\frac M {C^{i}}\}$. If $\sum_{i\in E} \int g_i=\infty$ then for almost every $x$, 
\[\underset{N\to \infty}{\lim}\, \frac{\sum_{i \in E}^N g_i(x)}{\sum_{i \in E}^N\int g_i}=1.\]
\end{prop}

We first state the appropriate version of approximate $T$-invariance for the $g_i$, an analogue of Lemma \ref{lem:almost invar}.

\begin{lem}\label{lem:almost invar2} 
For all $l<i$,
\[\max_{k<C^{a}} \, \Vert g_i-g_i \circ T^k\Vert_1 < \frac{4}{C-1}C^{a-l} \|g_l\|_1.\]
\end{lem}

\begin{proof}
Exactly as in the proof of Lemma \ref{lem:almost invar}, to bound $\|g_i-g_i\circ T^s\|_1$ we need to bound $8sr_{C^i}$. It is easy to bound $\|g_l\|_1\geq 2C^l(C-1)r_{C^{l+1}}\geq2C^l(C-1)r_{C^i}$. Letting $s=C^a$, we get the desired results after a quick computation.
\end{proof}

\begin{proof}[Proof of Proposition \ref{prop:reasonable i}] 
Suppose that $\sum_{i\in E} \int g_i = \infty$. Write $E = \{a_1 < a_2 < \cdots \}$. The idea of the proof is to show that $H_i=g_{a_i}$ satisfy the conditions (H1), (H2) and (H3$'$) of Proposition \ref{prop:abstract crit} from which the result follows (see Remark \ref{rem:h3 prime}). Recall: 
\[ \mbox{(H3$'$): }  \left| \sum_{j=1}^{N-1} \sum_{i=j+1}^N \int H_i H_j - \int H_i \int H_j \right| < C_2 \sum_{j=1}^{N-1} \|H_j\|_1.  \]
We replace (H3) with (H3$'$) since we cannot appeal to Lemma \ref{lem:L1 bound}.

By our assumption on $r_{C^{a_i}}$ and Lemma \ref{lem:hit bound} (see Remark \ref{rem:no khinchin}) we have $\Vert g_{a_i}\Vert_{\infty}<1+2M\sigma^{-1}$ and so condition (H1) is satisfied.

Condition (H2) is one of our assumptions.

Condition (H3$'$) follows from the proof of Proposition \ref{prop:correlation}, but requires a few modifications. $C^{a_i}$ play the role of $n_i$ and we let $u'_j=\max(\{j+1\} \cup \{i: C^{a_i}<\frac{1}{r_{C^{a_j}}}\})$. Let $K$ be chosen for $\sigma$ as in Lemma \ref{lem:dense2} and $v'=\log_C(K)+\log_C(2^{\hat{q}})$. Dividing up our sum as before we have:
\begin{align}
	\sum_{i={j+1}}^N|\int g_{a_i}g_{a_j}&-\int g_{a_j} \int g_{a_i}| = \sum_{j+1 < i \leq u'_j }^N|\int g_{a_i}g_{a_j}-\int g_{a_j} \int g_{a_i}| \nonumber \\
		& +\sum_{i=j+1 \ or \ u'_j<i\leq u'_j+4v'}^N |\int g_{a_i} g_{a_j}-\int g_{a_i} \int g_{a_j}| \nonumber \\
		& +\sum_{i>u'_j+4v'}^N |\int g_{a_i}g_{a_j}-\int g_{a_j} \int g_{a_i}|. \nonumber
\end{align}

Bounding the first sum by a constant multiple of $\Vert g_{a_j} \Vert_1$ follows the argument of Proposition \ref{prop:correlation}, Step 1. The argument requires only a bound of the type $e_T(n)>\frac{\sigma}{n}$, which we have, and the argument to extend Lemma \ref{lem:hit bound} as in the above proof of (H1). After summing over $1\leq j \leq N-1$, this portion of the sum satisfies (H3$'$).

Bounding the second term by a constant multiple of $\Vert g_{a_j} \Vert_1$ is also a direct application of Proposition \ref{prop:correlation}, Step 2. It suffices to show that there exists $\tilde{C} $ so that $\|g_{a_{j+1}}\|_{\infty} + \sum_{i=u'_j}^{u'_j+4v'} \|g_{a_i}\|_{\infty}<\tilde{C}$.  Because $v'$ is a constant, it suffices for  $\|g_{a_i}\|_\infty$ to be uniformly bounded, which follows from (H1). After summing over $1\leq j \leq N-1$, this portion of the sum satisfies (H3$'$).

The third sum requires the most care. Given $i$, for each $k > 4v'$ let 
\[b_{i,k} = \min \{ j : i = u'_j+k \} \ \mbox{ and } \ d_{i,k} = \max \{ j : i=u'_j +k\} \] 
and set $b_{i,k}=d_{i,k}=0$ if no $u'_j$ equals $i-k$. With these definitions, all $j$ between $b_{i,k}$ and $d_{i,k}$ have the same value of $u'_j$, namely $u'_{b_{i,k}}=u'_{d_{i,k}}$. Let $h_{i,k} = \sum_{l=b_{i,k}}^{d_{i,k}} g_{a_l}$. Note that if $b_{i,k}=d_{i,k}=0$, then $h_{i,k}=0$. 

First, we give a bound on $\|h_{i,k}\|_\infty$.

\begin{slem}
For all $k$, $\| h_{i,k} \|_\infty$ is uniformly bounded, independent of $i$ and $k$, by a constant depending only on $C$ and $\sigma$.
\end{slem}

\begin{proof}[Proof of Sublemma:]
We need only consider situations where $b_{i,k}$ and $d_{i,k}$ are nonzero. 

As $i$ and $k$ are fixed within this proof, to simplify notation below let us write $b=b_{i,k}$, $d=d_{i,k}$ and $h=h_{i,k}$. Then
\[ h(x) =  \sum_{l=b}^d g_{a_l}(x) = \sum_{l=b}^{d} \sum_{j=C^{a_l}}^{C^{a_l+1}-1} \chi_{B(\frac{1}{2},r_j)}(T^jx).\]
Following the argument of Lemma \ref{lem:hit bound}, we find that 
\begin{align} 
	\sum_{l=b}^{d} \sum_{j=C^{a_l}}^{C^{a_l+1}-1} \chi_{B(\frac{1}{2},r_j)}(T^jx) & \leq \lceil \frac{2r_{C^{a_b}}}{\sigma/C^{a_d+1}}\rceil \nonumber \\
		& \leq 1+ \frac{2C}{\sigma} r_{C^{a_b}} C^{a_d} \nonumber \\
		& \leq 1+ \frac{2C}{\sigma} r_{C^{a_b}} C^{a_{u'_d}} \nonumber \\
		& = 1+ \frac{2C}{\sigma} r_{C^{a_b}} C^{a_{u'_b}} \nonumber \\
		& < 1+ \frac{2C}{\sigma} r_{C^{a_b}} \frac{1}{r_{C^{a_b}}}  = 1+ \frac{2C}{\sigma} \nonumber 
\end{align}
using the definition of $u'_b$ at the last step. This proves the Sublemma.
\end{proof}

Let $i>u'_j+4v'$ be fixed, and write $i = u'_j+k$. We want to bound $|\int g_{a_i} h_{i,k} - \int g_{a_i} \int h_{i,k} |$. We continue to follow the argument of Step 3 of Proposition \ref{prop:correlation}. Let $b' = i-\frac{3k}{4}$ and let $S_{i,k}$ be the set of discontinuities of $T^{C^{a_{b'}}}$. By Lemma \ref{lem:dense2}, $S_{i,k}$ is $\frac{K}{C^{a_{b'}}}$-dense. 

We apply Lemma \ref{lem:const on blocks} as before to obtain $f_{i,k}$ approximating $h_{i,k}$. Then $\|f_{i,k}\|_1\leq \|h_{i,k}\|_1$, $\|f_{i,k}\|_\infty\leq \|h_{i,k}\|_\infty$, and with some short calculation,
\begin{equation}\label{eqn:bound a}
	 \Vert f_{i,k} - h_{i,k} \Vert_1 \leq 2 C^{a_{b_{i,k}}+1} \frac{K}{C^{a_{b'}}} = 2CC^{a_{b_{i,k}}}\frac{K}{C^{a_{u'_j+\frac k 4}}} \leq \tilde{K}  C^{-\frac{k}{4}} \Vert h_{i,k} \Vert_1
\end{equation}
for some uniform $\tilde{K}>0$. We have used the definition of $u'_j$, including the fact that $u'_j>b_{i,k}$, to bound $\Vert h_{i,k} \Vert_1 \geq 2C^{a_{b_{i,k}}}r_{C^{a_{b_{i,k}+1}}} \geq 2C^{a_{b_{i,k}}}r_{C^{a_{u'_j}}} \geq \frac{2C^{a_{d_{i,k}}}(C-1)}{C^{a_{u'_j+1}}}$.

We apply Theorem \ref{thm:quant bc} as before. For any $C^{a_{b'}}$-block interval $J$ and any $x$,
\[ \left|  |\{ 0 < k \leq C^{a_{i-k/4}} : T^kx\in J \}| -  C^{a_{i-k/4}} |J| \right| < C^{a_{i-k/4}} C_1'e^{-\tilde C_2 \frac{k}{2}} |J|\]
where $\tilde{C}_2=\frac{C_2}{\log_C(2)}$. Proceeding precisely as in Proposition \ref{prop:correlation} and the proof of Proposition \ref{prop:limsup}, we obtain
\begin{align}
	\left| \int g_{a_i} f_{i,k} - \int g_{a_i} \int f_{i,k} \right| & \leq \|f_{i,k}\|_1 C_1' e^{-\tilde C_2 \frac{k}{2}} \|g_{a_i}\|_1 + \|f_{i,k}\|_\infty \tilde c_{i,k}  \nonumber
\end{align}
where
\[ \tilde c_{i,k} = \max \{  \Vert g_{a_i} - g_{a_i} \circ T^k \Vert_1 :  0 \leq k \leq C^{a_{i-k/4}} \}.  \]

We apply Lemma \ref{lem:almost invar2} with $a=a_{i-k/4}$ and $l=a_{i-1}$ and obtain that $\tilde c_{i,k} < \frac{4}{C-1}C^{a_{i-k/4}-a_{i-1}}\|g_{a_{i-1}}\|_1 \leq \frac{4C}{C-1}C^{-\frac{k}{4}}\|g_{a_{i-1}}\|_1$.
Therefore, using that $\|f_{i,k}\|_1 \leq \|f_{i,k}\|_\infty \leq \|h_{i,k}\|_\infty$ and $\|g_{a_j} \|_\infty$ are universally bounded,
\begin{equation}\label{eqn:bound b}  
	| \int g_{a_i} f_{i,k} - \int g_{a_i} \int f_{i,k} | \leq \hat D_1  e^{-\tilde C_2 \frac{k}{2}} \|g_{a_i}\|_1 + \hat D_2 C^{-\frac{k}{4}} \|g_{a_{i-1}}\|_1.
\end{equation}

We now follow the end of the proof of Proposition \ref{prop:correlation}. The exponential decay in equations \eqref{eqn:bound a} and \eqref{eqn:bound b} and the universal bound on $\|g_{a_i}\|_\infty$ allows us to show that
\[ \sum_{i=1}^N \sum_{k=1}^{N-1} \left| \int g_{a_i} h_{i,k} - \int g_{a_i} \int h_{i,k}\right| \leq C_2 \sum_{i=1}^N \|g_{a_i}\|_1.\]
It is straightforward to check that this implies (H3$'$) with $H_i=g_{a_i}$, as desired.
\end{proof}

%

\subsection{\boldmath{$ir_i$} big} \label{sec:big section}

When $ir_i\geq M$ we want to use the next lemma, which requires $M$ sufficiently large:

\begin{lem} \label{lem:big r} 
Let $T$ be of constant type and $C>1$. Then, uniformly in $a\in[0,1]$,
\[\lim_{M\to\infty} \ \limsup_{j\to \infty} \ \sup_x \left| \left( \frac{C^{j+1}}{2M(C^{j+1}-C^j)}\sum_{i=C^j}^{C^{j+1}-1}\chi_{B(a,\frac M {C^{j+1}})}(T^ix)\right) -1\right|=0.\]
\end{lem}

\begin{proof}

Fix $\epsilon>0$. Fix $C$ and a value of $k$ to be chosen later. Because $T$ is of constant type, for any choice of $k$, for sufficiently large $M$ (which depends on $k$), any interval $B(a,\frac{M}{C^{j+1}})$ can be approximated up to an $\epsilon$ proportion by $C^{j-k}$-blocks (of $T$), for $j$ sufficiently large (independent of $a$). The remainder of the proof is determining how large $k$ needs to be. 

We now choose $n_i=3^i$ in the statement of Theorem \ref{thm:quant bc} with $c=\sigma$.
By Theorem \ref{thm:quant bc}, by choosing $Q$ large enough (given $C,\sigma$ and $\frac \epsilon 4$) we have that if $n_r$ is the largest $n_i<C^{j-k+Q}$  and  $\hat{J}$ is any $C^{j-k}$-block we have 
\begin{equation}\label{eq:chunk estimate}
	\left|\frac{1}{n_r}\sum_{i=1}^{n_r}\chi_{\hat{J}}(T^iT^{C^j}x)-|\hat{J}|\right|<\epsilon/4 \ \ \mbox{ for all } x.
\end{equation}
Note that $Q$ may be chosen independent of $k$.

First, choose $k$ so large that $C^{j-k+Q}<C^{j+1}-C^j$. Let $d=\left\lfloor \frac{C^{j+1}-C^j}{n_r}\right\rfloor > 0$. Decompose the sum in the lemma into $d$ sums over $n_r$ indices each, together with a remainder sum of length $<n_r$. Applying inequality \eqref{eq:chunk estimate} to the length-$n_r$ sums, we obtain

\begin{align}\label{eq:b2}
	\Big| \frac{1}{C^{j+1}-C^j} & \sum_{i=1}^{C^{j+1}-C^j}  \chi_{\hat{J}}(T^iT^{C^j-1}x)-|\hat{J}| \Big| \nonumber \\
	& =\frac 1 {C^{j+1}-C^j}\left(\sum_{\ell=0}^{d-1}\sum_{i=\ell n_r+C^j}^{(\ell+1)n_r+C^j}\chi_{\hat{J}}(T^ix)+\sum_{i=dn_r+C^j}^{C^{j+1}}\chi_{\hat{J}}(T^ix) \right)-|\hat{J}| \nonumber \\
	& \leq \frac{2}{d} \Big| \sum_{\ell=0}^{d-1} \left( \frac{1}{n_r} \sum_{i=\ell n_r+C^j}^{(\ell+1)n_r +C^j-1} \chi_{\hat J}(T^ix) -|\hat J| \right) +\frac{1}{dn_r} \sum_{i=dn_r C^j}^{C^{j+1}-1} \chi_{\hat J} (T^ix) \Big| \nonumber \\
	& \leq \frac{2}{d} d \frac{\epsilon}{4} + \frac{2}{dn_r} \max_y \sum_{i=1}^{n_r} \chi_{\hat J} (T^iy) \leq \frac{\epsilon}{2}+\frac{2}{d}.
\end{align}

Similarly, let $U$ be the subset of $B(a,\frac{M}{C^{j+1}})$ that is not made up of $C^{j-k}$ blocks. It is at most 2 intervals. By the constant type assumption, we bound
\begin{equation}\label{eq:b4}
\left|\sum_{i=C^j}^{C^{j+1}-1}\chi_U(T^ix)\right| \leq (C^{j+1}-C^j)\sigma^{-1}|U|+1
\end{equation}
for all $x$, independent of $a$. Given any choice of $k$, we choose $M$ large enough at the beginning to make $|U|<\epsilon \frac{2M}{C^{j+1}}$, controlling the contribution of inequality \eqref{eq:b4}.

The lemma now follows if we can choose $k$ and $M$ large enough to make equation \eqref{eq:b2} less than $\epsilon$. This is clear as, for large $j$, by taking $k$ large we can ensure $n_r$ is small compared with $C^{j+1}-C^j$, and therefore that $d$ is large. This completes the proof. 
 \end{proof}

The next lemma lets us split up the natural numbers into subsets where we appeal to Proposition \ref{prop:reasonable i}, subsets where we can apply Lemma \ref{lem:big r} (see Lemma \ref{lem:G bound}), and a small remaining piece that we show is negligible (see Corollary \ref{cor:B bound}).

Throughout the remainder of this section $C>1$ should be thought of as very close to 1. Define
\[G_{C,\rho,M}=\left\{j \in \mathbb{N}:r_{ C^{j+1}}\geq \frac M {C^{j+1}} \text{ and } \rho r_{C^j} \leq r_{C^{j+1}}\right\}\]
and 
\[B_{C,\rho,M}=\left\{j \in \mathbb{N}\setminus G_{C,\rho,M} :r_{ C^{j+1}}\geq \frac M {C^{j+1}}\right\}.\]

When $\rho$ is very close to 1, $G_{C,\rho,M}$ is the set of indices where Cauchy condensation (that is, replacing $r_i$ with $r_{C^{j+1}}$ for $C^j<i\leq C^{j+1}$) is a mild change in the size of radii. 

\begin{lem}\label{lem:partition}  
For any $\epsilon>0$ and any $\rho<1$, there exists $C>1$ so that for any non-increasing sequence $\{r_i\}\subset \mathbb{R}^+$, we have
\[\limsup_{N \to \infty} \frac{\sum_{ j \in B_{C,\rho,M}: C^{j+1}<N}\, (C^{j+1}-C^j)r_{C^j}}{\sum_{i=1}^N\, r_i}<\epsilon\] for all $M>2\max\{1,r_1\}$.
\end{lem}

\begin{proof}
Let $\epsilon>0$ and $\rho<1$ be given. By assumption $r_1<\frac{M}{2}$. 

Enumerate $B_{C,\rho,M}=\{b_1, b_2, \ldots\}$ in increasing order. 
\\

\noindent\textit{Claim:} $b_n \geq n \log_C(\frac{1}{\rho}) + \log_C(2)-1$.

\textit{Proof of claim:} By definition of $B_{C,\rho,M}$, each new $b_i$ decreases $r_{C_j}$ by a factor of at least $\rho$. Since $r_{C^{b_1}}<\frac{M}{2}$, this implies that $r_{C^{b_n}} < \frac{M}{2} \rho^{n-1}$. Therefore, using again the definition of $B_{C,\rho,M}$, 
\[ \frac{M}{C^{b_n+1}} \leq r_{C^{b_n+1}} \leq \rho r_{C^{b_n}} < \frac{M}{2} \rho^n.\]
Taking $\log_C$ of both sides yields the claim.
\\

Let $S_0 = \emptyset$. Define $S_k$ inductively by letting $S_{k+1}$ be the $d:=\left\lceil \frac{1}{2} \log_C(\frac{1}{\rho})\right\rceil$ largest indices in 
\[\{1,2,\ldots b_{k+1}\} \setminus \left(B_{C,\rho,M} \cup \bigcup_{i=1}^kS_i\right).\]
The claim above ensures that, for any choice of $\rho$ if $C>1$ is small enough, such a set exists.

To prove the Lemma, it clearly suffices to show that for all small enough $C>1$, for all sufficiently large $k$, we have
\begin{equation}\label{eq:sum above}
\epsilon \sum_{j\in S_k} \sum_{i=C^j}^{C^{j+1}-1}2r_i>2(C^{b_k+1}-C^{b_k})r_{C^{b_k}}.
\end{equation}

First, we choose $C>1$ such that $C<\frac{1}{\rho}$. Write $S_k = \{ u_1 > u_2 > \cdots > u_d \}.$ Then,
\[ \sum_{j=1}^d \sum_{i=C^{u_j}}^{C^{u_j +1}-1} 2r_i \geq \sum_{j=1}^d 2 r_{C^{u_j+1}} C^{u_j} (C-1).\]
Suppose that $m_j \geq 0$ of the $b_i$ lie in $[u_j+1,b_k)$. Then, from the definition of $B_{C,\rho,M}$, $r_{C^{u_j+1}} > (\frac{1}{\rho})^{m_j} r_{C^{b_k}}$ and $u_j = b_k-j-m_j$. Applying this to the bound above, we have 
\[ \sum_{j=1}^d \sum_{i=C^{u_j}}^{C^{u_j +1}-1} 2r_i \geq \sum_{j=1}^d 2 \left(\frac{1}{\rho}\right)^{m_j} r_{C^{b_k}} C^{b_k-j}\left(\frac{1}{C}\right)^{m_j}(C-1).\]
Then, using the assumption $C<\frac{1}{\rho}$ and so $(\frac 1 \rho)^{m_j}(\frac 1 C)^{m_j}>1$ and carrying out the sum and using the definition of $d$, we find that 
\begin{equation}\label{denom bound}
	\sum_{j=1}^d \sum_{i=C^{u_j}}^{C^{u_j +1}-1} 2r_i \geq 2 r_{C^{b_k}} C^{b_k}(1-C^{-d}) \geq 2 r_{C^{b_k}} C^{b_k}(1-\sqrt{\rho}).
\end{equation}

If we pick $C>1$ so that
\[ \epsilon 2r_{C^{b_k}} C^{b_k}(1-\sqrt{\rho}) > 2r_{C^{b_k}} C^{b_k}(C-1)\] (which is clearly possible) then inequality \eqref{denom bound} shows that inequality \eqref{eq:sum above} establishes the lemma.
\end{proof}

To control $\sum_{k=C^j}^{C^{j+1}-1} \chi_{B(\frac{1}{2} , r_k)}(T^kx)$ where $j \in B_{C,\rho,M}$ we need the following result.

\begin{lem}\label{lem:control bad} 
Let $T$ be an IET of constant type, $\{r_i\}$ nonincreasing.  Then for all $x$,  
\[\sum_{i=C^j}^{C^{j+1}-1} \chi_{B(\frac{1}{2} ,r_i)}(T^ix)<\frac{2r_{C^j}}{\sigma}(C^{j+1}-C^j)+1.\]
\end{lem}

The proof of this Lemma is essentially the same as the proof of Lemma \ref{lem:hit bound}.

\begin{cor}\label{cor:B bound} 
For every $\epsilon>0$ and $\rho<1$ there exists $C$ so that for all $x$, and all large enough $M$ we have 
\[\frac{\sum_{j\in B_{C,\rho,M}}^{N'} \sum_{i=C^j}^{C^{j+1}-1}\chi_{B(\frac{1}{2}, r_i)}(T^ix)}{\sum_{j =1}^{N'}\sum_{i= C^j}^{C^{j+1}-1}2r_i}<\epsilon\]
for sufficiently large $N'$.
\end{cor}

\begin{proof}
By Lemma \ref{lem:control bad}, for all $j$,
\[  \sum_{i=C^j}^{C^{j+1}-1}\chi_{B(\frac{1}{2}, r_i)}(T^ix) < \frac{2 r_{C^j}}{\sigma} C^j (C-1)r_{C^j}+1. \]
Using this fact, for all $x$ we have,
\[ \frac{\sum_{j\in B_{C,\rho,M}}^{N'} \sum_{i=C^j}^{C^{j+1}-1} \chi_{B(\frac{1}{2},r_i)}(T^ix)}{\sum_{j=1}^{N'} \sum_{i=C^j}^{C^{j+1}-1}2r_i} < 
		\frac{\sum_{j\in B_{C,\rho,M}}^{N'} \frac{2r_{C^j}}{\sigma}C^j(C-1)+1}{\sum_{j=1}^{N'} \sum_{i=C^j}^{C^{j+1}-1}2r_i}.\]
Note that for $j\in B_{C,\rho,M}$, we have $C^j r_{C^j} \geq \frac{M}{C\rho}$. Therefore, for $M$ sufficiently large (say, $>(C-1)^{-1}$),we have 
\[\frac{\sum_{j\in B_{C,\rho,M}}^{N'} {1}}{\sum_{j=1}^{N'} \sum_{i=C^j}^{C^{j+1}-1}2r_i}=
	O\left(\frac{\sum_{j\in B_{C,\rho,M}}^{N'} \frac{2r_{C^j}}{\sigma}C^j(C-1)}{\sum_{j=1}^{N'} \sum_{i=C^j}^{C^{j+1}-1}2r_i}\right).\]
Therefore, it is sufficient to bound $\frac{\sum_{j\in B_{C,\rho,M}}^{N'} \frac{2r_{C^j}}{\sigma}C^j(C-1)}{\sum_{j=1}^{N'} \sum_{i=C^j}^{C^{j+1}-1}2r_i}$.
We apply Lemma \ref{lem:partition} with $N=C^{N'+1}-1$ to this and obtain that for all sufficiently large $M$ and $N'$, 
 $\frac{\sum_{j\in B_{C,\rho,M}}^{N'} \frac{2r_{C^j}}{\sigma}C^j(C-1)}{\sum_{j=1}^{N'} \sum_{i=C^j}^{C^{j+1}-1}2r_i}$ is bounded by some fixed multiple of $\epsilon$, proving the result.
\end{proof}

\begin{lem}\label{lem:G bound} 
For any $\epsilon>0$ and $C>1$ there exists $M_0>1$ so that if $M>M_0$ and $\rho=1-\frac{\epsilon^2 \sigma}{4}$ then for all sufficiently large $j \in G_{C,\rho,M}$ and all $x$,
\[\frac{\sum_{i=C^j}^{C^{j+1}} \chi_{B(\frac 1 2 ,r_i)}(T^ix)}{\sum_{i=C^j}^{C^{j+1}}2r_i}\in [1-\epsilon,1+\epsilon].\]
\end{lem}

\begin{proof}
Fix $\epsilon>0$. We may assume that $\epsilon<\frac{3}{8}$ and that $\sigma<1$. First, note that to prove the Lemma it is sufficient to show that for sufficiently large $j \in G_{C,\rho,M},$
\[\frac{\sup_x \sum_{i=C^j}^{C^{j+1}} \chi_{B(\frac{1}{2},r_i)}(T^i x)}{\inf_x \sum_{i=C^j}^{C^{j+1}} \chi_{B(\frac{1}{2},r_{C^{j+1}})}(T^i x) } \leq 1+\epsilon.\]

By Lemma \ref{lem:big r} we have  that if $M$ is large enough
\[\frac{\sup_x \sum_{i=C^j}^{C^{j+1}} \chi_{B(\frac{1}{2},r_{C^{j+1}})}(T^i x)}{\inf_x \sum_{i=C^j}^{C^{j+1}} \chi_{B(\frac{1}{2},r_{C^{j+1}})}(T^i x) }\leq1+\epsilon,\] 
and so it suffices to show that for any $\epsilon>0,$ 
\begin{equation}\label{close1}
	\frac{\sup_x \sum_{i=C^j}^{C^{j+1}} \chi_{B(\frac{1}{2},r_i)\setminus B(\frac{1}{2},r_{C^{j+1}})}(T^i x)}{\inf_x \sum_{i=C^j}^{C^{j+1}} \chi_{B(\frac{1}{2},r_{C^{j+1}})}(T^i x) }\leq\epsilon.\end{equation}

First, we bound the numerator of \eqref{close1}. Consider $B(\frac{1}{2} ,r_i)\setminus B(\frac{1}{2} ,r_{C^{j+1}})$ for $i\geq C^j$. It consists of two intervals of size at most $(1-\rho)r_{C^j}$ since $j\in G_{C,\rho,M}$. By Lemma \ref{lem:control bad} and our choice of $\rho$,
\begin{align}
	\sup_x \sum_{i=C^j}^{C^{j+1}} \chi_{B(\frac{1}{2},r_i)\setminus B(\frac{1}{2},r_{C^{j+1}})}(T^i x) &\leq 2+2(1-\rho)r_{C^j}2\frac{(C^{j+1}-C^j)}{\sigma} \nonumber \\
\label{eq:array}	&= 2+ \frac{\epsilon^2}{2} 2r_{C^j}(C^{j+1}-C^j). 
\end{align}

To bound the denominator of \eqref{close1} below we appeal to Lemma \ref{lem:big r}.  First, let $M_0'$ be so large that for all $a\in [0,1]$ and all $x$, 
\[ \sum_{i=C^j}^{C^{j+1}}  \chi_{B(a,\frac{M_0'}{C^{j+1}})}(T^ix) \geq \left(1-\frac{\epsilon}{3}\right) 2 \frac{M_0'}{C^{j+1}}(C^{j+1}-C^j) \]
for sufficiently large $j$ (independent of $a$). Let $M_0 = \max\{3 M_0',4\epsilon^{-1}C(C-1)^{-1}\}$. We consider $j\in G_{C,\rho,M}$ with $M>M_0$, $r_{C^{j+1}} \geq 3\frac{M_0'}{C^{j+1}}$. Partition $B(\frac{1}{2},r_{C^{j+1}})$ into $\lambda\geq3$ intervals of size $\frac{2M_0'}{C^{j+1}}$ and one interval of size $<\frac{2M_0'}{C^{j+1}}$. Let $B$ be the union of the $\lambda$ intervals. Applying Lemma \ref{lem:big r} as above to each of the $\lambda$ intervals forming $B$ we obtain, for sufficiently large $j\in G_{C,\rho,M}$,
\begin{align} 
	\sum_{i=C^j}^{C^{j+1}} \chi_{B(\frac{1}{2},r_{C^{j+1}})}(T^ix) & \geq \sum_{i=C^j}^{C^{j+1}} \chi_{B}(T^ix) \nonumber \\
		& \geq \left(1-\frac{\epsilon}{3}\right) \lambda \frac{2M_0'}{C^{j+1}} (C^{j+1}-C^j) \nonumber
\end{align}

Further, because $(\lambda+1) \frac{2M_0'}{C^{j+1}}>2r_{C^{j+1}}$ this is	
\begin{align}
	\phantom{OOOOOOOOOOOOOOOOO}&  \geq \left(1-\frac{\epsilon}{3}\Big) \Big(\frac{\lambda}{\lambda+1}\right) 2r_{C^{j+1}} (C^{j+1}-C^j). \nonumber \\
		&  \geq \left(1-\frac{\epsilon}{3}\right) \left(\frac{3}{4}\right) 2r_{C^{j+1}} (C^{j+1}-C^j). \nonumber \\
\label{eq:array2}		& \geq\frac{1}{2} 2r_{C^{j+1}}(C^{j+1}-C^j). 
\end{align}
since $\epsilon<1.$

Combining inequalties \eqref{eq:array} and \eqref{eq:array2}, 
\[ \frac{\sup_x \sum_{i=C^j}^{C^{j+1}} \chi_{B(\frac{1}{2},r_i)\setminus B(\frac{1}{2},r_{C^{j+1}})}(T^i x)}{\inf_x \sum_{i=C^j}^{C^{j+1}} \chi_{B(\frac{1}{2},r_{C^{j+1}})}(T^i x) }  <\frac{1+ \frac{\epsilon^2}{2}   r_{C^j}(C^{j+1}-C^j)}{\frac{1}{2}  r_{C^{j+1}}(C^{j+1}-C^j)}. \]
Now, $\frac{2}{ r_{C^{j+1}}(C^{j+1}-C^j)} \leq \frac{\epsilon}{2}$ using our choice of $M_0\geq 4\epsilon^{-1}C(C-1)^{-1}$. Also, $\frac{\epsilon^2 r_{C^j}(C^{j+1}-C^j)}{ r_{C^{j+1}}(C^{j+1}-C^j)} \leq \frac{\epsilon^2}{\rho} \leq \frac{4}{3}\epsilon^2 \leq \frac{\epsilon}{2}$ using the fact that $j\in G_{C, \rho,M}$, our choice of $\rho$, and the fact that $\epsilon< \frac{3}{8}.$ This completes the proof.
\end{proof}

We note the following facts about the results above. First, $\rho$ does not depend on $C$ and so we may choose $C$ for Lemma \ref{lem:partition} to hold. Also our only requirement on $M$ in Corollary \ref{cor:B bound} and Lemma \ref{lem:G bound} is that it is large enough. So given $\rho,C$ we may choose (a possibly larger) $M$ so that Corollary \ref{cor:B bound} and Lemma \ref{lem:G bound} hold.

We are now ready to prove Theorems \ref{constant type} and \ref{thm:iet bad approx}.

\begin{proof}[Proof of Theorems \ref{constant type} and \ref{thm:iet bad approx}]

It suffices to show that for all $\delta>0$ there exists $C>1$ so that 
\[ \liminf_{N \to \infty} \frac{ \sum_{j=1}^N \sum_{i=C^j}^{C^{j+1}}\chi_{B(\frac{1}{2} ,r_i)}T^ix}{\sum_{j=1}^N \sum_{i=C^{j}}^{C^{j+1}} 2r_i}>1-\delta\]
and 
\[\limsup_{N \to \infty} \frac{ \sum_{j=1}^N \sum_{i=C^j}^{C^{j+1}}\chi_{B(\frac{1}{2} ,r_i)}T^ix}{\sum_{j=1}^N \sum_{i=C^{j}}^{C^{j+1}} 2r_i}<1+\delta .\]
Choose $\epsilon=\frac \delta 2$ and $\rho=1-\frac{\epsilon^2\sigma}{4}$. Following Corollary \ref{cor:B bound}, choose $C$ for this $\rho$ and $\epsilon$. Following Lemma \ref{lem:G bound} and Corollary \ref{cor:B bound}, choose $M$ for these $\rho, C, \epsilon$. Then by Lemma \ref{lem:G bound} we have
\[\limsup_{N \to \infty}\frac{\sum_{j \in G_{C,\rho,M}}^N \sum_{i=C^j}^{C^{j+1}}\chi_{B(\frac{1}{2}, r_i)}(T^i x)}{\sum_{j\in G_{C,\rho,M}}^N\sum_{i=C^j}^{C^{j+1}}2r_i}<1+\frac{\delta}{2}\]
and
\[\liminf_{N \to \infty}\frac{\sum_{j\in G_{C,\rho,M}}^N \sum_{i=C^j}^{C^{j+1}}\chi_{B(\frac{1}{2}, r_i)}(T^i x)}{\sum_{j \in G_{C,\rho,M}}^N\sum_{i =C^j}^{C^{j+1}}2r_i}>1-\frac{\delta}{2}.\]

Proposition \ref{prop:reasonable i} implies 
\[\lim_{N \to \infty} \frac{\sum_{j \notin (G_{C,\rho,M}\cup B_{C,\rho,M})}^N\sum_{i=C^j}^{C^{j+1}}\chi_{B(\frac{1}{2} , r_i)}(T^ix)} {\sum_{j \notin (G_{C,\rho,M}\cup B_{C,\rho,M})}^N\sum_{i=C^j}^{C^{j+1}}2r_i}=1\]
for almost every $x$. By Corollary \ref{cor:B bound} 
\[\limsup_{N \to \infty} \frac{\sum_{j\in B_{C,\rho,M}}^N \sum_{i=C^j}^{C^{j+1}}\chi_{B(\frac{1}{2}, r_i)}(T^i x)} {\sum_{i=1}^{C^{N+1}}2r_i}<\frac{\delta}{2}\]  for all $x$, which completes the proof.
\end{proof}

%

\section{Quantitative Boshernitzan's criterion}\label{sec:quant bc}

This section uses Appendix \ref{symb code}. In that Appendix, we recall that $T:[0,1) \to [0,1)$ is measure conjugate to a subshift $S:X \to X$ of the full shift on $d$ symbols. In this section we use both of these (measure-theoretically) equivalent descriptions of the dynamics for various proofs, as suits our purposes.

We want to prove a quantitative version of the following:

\begin{thm} [Boshernitzan \cite{bosh_crit}]\label{thm:bosh crit} 
Let $S:X\to X$ be the left shift acting minimally on a symbolic dynamical system. Let $\mu$ be an $S$-invariant measure.  Let $\epsilon_n$ be the $\mu$ measure of the smallest cylinder set of length $n$. If there exists a constant $c$ such that for infinitely many $n$, $\epsilon_n \geq \frac{c}{n}$, then $S$ is $\mu$-uniquely ergodic.
\end{thm}

An analogue of this result was proved for IETs by Veech \cite{V bosh crit}, in which case the invariant/ergodic measure is Lebesgue. Masur \cite{masur crit} established the analogous, in fact stronger, result for flows on flat surfaces.

We will prove Theorem \ref{thm:quant bc}, stated in the introduction for $T$, for the shift $S$ measure conjugate to $T$ described in Appendix $A$:

\begin{thm} \label{thm:shift version}
Let $S:X\to X$ be the symbolic system for a minimal IET, $\mu$ be an invariant measure and $\epsilon_n$ be the smallest $\mu$-measure of an $n$-cylinder of $S$. Assume there exists some $c>0$ and a sequence $(n_i)$ with $n_i\geq10n_{i-1}$ such that $\epsilon_{n_i}>\frac{c}{n_i}$ for all $i$. Let $w$ be a word of length $n_i$ and let $\chi_w$ be the characteristic function for the cylinder set defined by $w$. Then there exist positive constants $C_1,\, C_2, \, \hat{q}$ depending only on $c$ such that for all $x,x'\in X$ we have 
\[\frac {1}{n_{i+\hat{q}+L}}\left|\sum^{n_{i+\hat{q}+ L}}_{j=1} \chi_w(S^jx)-\chi_w(S^jx')\right|<C_1e^{-C_2L} \mu(w)\]
for all $L\in \mathbb{N}$. Here, $\mu(w)$ denotes the measure of the cylinder set defined by $w$.
\end{thm}

We note that in Theorem \ref{thm:quant bc} we assumed $n_i\geq 2n_{i-1}$. To see that Theorem \ref{thm:quant bc} and Theorem \ref{thm:shift version} are equivalent, we can pass to subsequences of the form $n_{4i+k}$.

This is a quantitative version of Boshernitzan's criterion because it tells how quickly any orbit equidistributes. Quantitative ergodicity statements for IETs and flows have been profitably studied with deep results in \cite{forni}, \cite{zorich} and  \cite{ath_forn}.

The next proposition is similar to results used in \cite{V bosh crit}. It provides a construction of a set of \emph{Rokhlin towers} describing the dynamics of $T$ which will be useful in the rest of our proof. Specifically, conditions (1), (2) and (3) define a set of Rokhlin towers $\{ (J_a, m_a)\}$ decomposing $[0,1)$. The rest of the proposition gives quantitative control over the number of towers $t$, the measures of the bases of the towers (and so the levels) $|J_a|$, and the \emph{heights} $m_a$ of the towers.

\begin{prop} \label{towers} 
If $e_T(2n)\geq\frac {c}{2n}$ then there exist intervals $J_1,\dots, J_t$  and numbers $m_1,\dots, m_t$ so that 
\begin{enumerate}
	\item $T^i J_a\cap T^j J_b=\emptyset$ for all $(i,a)\neq (j,b)$ with $0\leq i<m_a$ and $0\leq j<m_b,$
	\item $\cup_{a=1}^r \cup_{\ell=0}^{m_a-1}T^\ell J_a=[0,1),$
	\item $T^i$ is continuous (and therefore an isometry) on $J_a$ for all $0\leq i<m_a$,
	\item $|J_a|\geq e_T(2n)$,
	\item $t\leq \frac{2}{c}$,
	\item $n\leq m_a\leq 2n$ for all $a$.
\end{enumerate}
\end{prop}

\begin{proof}
Recall that $P_{S_k}$ is the partition of $[0,1)$ by the discontinuities of $T^k$. We denote $P_{S_k}$ by $P_k$. If $I\in P_k$, then $I$ has the form $[T^{-n_1}\delta_1, T^{-n_2}\delta_2)$, where $\delta_i$ are discontinuities of $T$ and $0\leq n_i \leq k-1$, and $T^k|_I$ is continuous.

We will construct the Rokhlin towers by drawing the $J_a$'s from the collections $P_n$ and $P_{2n}$. This will ensure that (3) and (4) are satisfied. Once $m_a$ are chosen satisfying (6), $[0,1)$ is the union of at least $n$ copies of each $J_a$. Since $|J_a|\geq e_T(2n) \geq \frac{c}{2n}$ the $m_a$ copies of $J_a$ cover a subset of $[0,1)$ of measure at least $\frac{c}{2}$. Once the disjointness described in (1) is assured, this implies that there are at most $\frac{2}{c}$ of the $J_a$'s, proving (5).

The rest of our proof uses the following simple claim:

\noindent \textbf{Claim:} \emph{If $I_1, I_2 \in P_k$, then $T^{l_1}I_1\cap T^{l_2}I_2 \neq \emptyset$ for some $0\leq l_1\leq l_2<k$ implies $T^{l_1}I_1 \subseteq T^{l_2}I_2$.}

\begin{proof}[Proof of claim:]
We may assume $l_1<l_2$, and then it is sufficient to prove the result for $l_1=0$ by applying $T^{-l_1}$.

As noted above $I_2= [T^{-n_1}\delta_1, T^{-n_2}\delta_2)$ for some $0\leq n_i \leq k-1$. Suppose that $I_1\cap T^{l_2} I_2 \neq \emptyset$ for some $0<l<k$. Unless $I_1 \subseteq T^{l_2}I_2$, we have $T^{l_2}(T^{-n_i}\delta_i) \in Int(I_1)$ for either $i=1$ or $2$. Then $Int(T^{n_i-l_2}I_1)$ contains the discontinuity $\delta_i$. But $n_i-l_2 \leq k-2$ and so $T^k|_{I_1}$ is not continuous, a contradiction.
\end{proof}

Let $J_1,\dots, J_r$ be a maximal subset of $P_n$ so that $T^i(J_a)\cap T^j(J_b)=\emptyset$ for all $(i,a)\neq (j,b)$ with $0\leq i,j<n$. (The claim applied with $I_1=I_2$, an element of  $P_n$ of minimal length, ensures that such a subset exists. Indeed continuity follows from the fact that $I_1$ is an element of $P_n$ and disjointness follows from Lemma \ref{interval}.) Let $m_a=n$ for $a=1,\ldots r$ and  $V_1=\cup_{a=1}^r\cup_{i=0}^{n-1}T^iJ_a$.

If $V_1=[0,1)$ we are done. Otherwise split $V_1^c$ into two sets:
\[U_A=\{x:\exists i<0<j \text{ so that }T^ix, T^jx \in V_1 \text{ and }j-i<n\}\]
\[U_B=(V_1\cup U_A)^c.\]

We now show that $U_A$ and $U_B$ are both unions of elements of $P_{2n}$. For each $x \in U_A$, consider the element $I$ of $P_n$ so that $x \in I$. We have $T^iI \cap J_a\neq \emptyset$ for some $0<i<n$ and some $1\leq a \leq r$. Moreover, $T^{-i}(T^i I \cap J_a)$ is a union of elements in $P_{2n}$. Therefore elements of $P_{2n}$ are either contained in $U_A$ or disjoint from it. Since elements of $P_{2n}$ are clearly either contained in $V_1$ or disjoint from it, similarly $U_B$ is a union of elements of $P_{2n}$.

Now we show how to cover $U_B$ as in the statement of the proposition. Let $I'_1,...,I'_u$ be the elements of $P_{2n}$ which are contained in $U_B$ and such that $T^{-1}I'_i \cap V_1\neq \emptyset$. By construction these also have $T^{-1}I'_i \cap \cup_{a=1}^rT^{n-1}J_a \neq \emptyset$. By the claim, this implies for each $i$ there exists $a$ so that  $T^{-1}I'_i \subset T^{n-1}J_a$. Now if $T^{m^*}I'_i \cap (\cup_{a=1}^r J_a) \neq \emptyset$ for some $m^*< 2n$ (which is necessarily at least $n$) we add $I'_i$ to our collection $\{J_a\}$, set the corresponding $m_a=m^*$, and we add $\cup_{\ell=0}^{m^*-1} T^\ell I_i'$ to $V_1$. Otherwise we add $I'_i$ to the $\{J_a\}$, set $m_a=n$, and add $\cup_{\ell=0}^{n-1}T^{\ell}I'_i$ to $V_1$. We call such an $I'_i$ \emph{recalcitrant}. Performing this for all of the $I'_i$ we obtain $V_2$. We now consider $I''_1,...,I''_v$ so that $I''_j$ are the elements of $P_{2n}$ whose pre-images are contained in $T^n(I_i')$ for some recalcitrant $I_i'$. As before we add the $I''_j$ to $\{J_a\}$ and, if $T^{m^*}I''_j \subset \cup_{a=1}^r J_a$ for some $m^*< 2n$, which is necessarily at least $n$, we set $m_a=m^*$. Otherwise we set $m_a=n$. We add all the $\cup_{\ell=0}^{m_a-1}T^\ell I''_j$ to $V_2$. In this way we obtain $V_3$. We repeat this procedure until we cannot continue, having obtained $V_k$. Observe $V_k$ is covered by a union of towers that satisfy (1), (3), (4), (5) and (6) and covers all of $V_1$ and $U_B$. Therefore anything missing is in $U_A$. We now treat these points.

Now we show how to cover $U_a$ as in the statement of the proposition. If $x \in V_k^c$ then there exist $I \in P_{2n}$ and $i,\ell \in \mathbb{N}$ so that $x\in T^i I$, $T^\ell I \subset J_a$ for some $a\in \{1,...,r\}$, $0\leq i<\ell<n$ and $T^{-1}I \cap V_1 \neq \emptyset$. As above, the claim implies that  $T^{-1}I \subset V_1$. Let $I_1,...,I_s$ be these $I$ and $j_i$ be so that $T^{j_i}I_i \subset I_a$ for some $a$.  If $I_1,...,I_q$ are the $I_i$ that orbit into $J_a$, we refine $\cup_{i=0}^{n-1}T^i J_a$ to be $(\cup_{i=1}^q \cup_{\ell=0}^{j_i+n}T^\ell I_i)\cup (\cup_{\ell=0}^{n-1}T^\ell(J_a\setminus (\cup_{i=1}^qT^{j_i}I_i)))$. Consider $J_a$ partitioned into elements of $P_{2n}$. By the claim, $\cup_{i=1}^qT^{j_i}I_i$ is a union of these partition elements and so its complement is as well. Therefore, replacing $J_a$ with $I_1, \ldots I_q$ (with corresponding $m_a=j_i+n$) and with the elements of $J_a\setminus (\cup_{i=1}^qT^{j_i}I_i)$ (with corresponding $m_a=n$) and using the $(J_a, m_a)$ defined in the argument above, we obtain in total a collection $\{(J_a,m_a)\}$ satisfying condition (2) in addition to the previously ensured (1), (3), (4), (5) and (6). This completes the proof.
\end{proof}

We also need the following results on symbolic systems:

\begin{lem} \label{cliques}
Let $S':X'\to X'$ be a symbolic system such that $\epsilon_1>c'$ and $\epsilon_{n'_i} \geq \frac{c'}{n'_i}$ for a sequence $n'_i$ such that $n'_i \geq 10n'_{i-1}$. Without loss of generality we assume that $c'<1$. Let $C_i=\{x:\sum_{j=0}^{n'_i-1}\chi_1(S'^jx)\geq \frac{c'^{2i+1}}{32^i}n'_i\}$. That is, $C_i$ is the set of all $x$ so that the symbol 1 occurs at least a proportion $\frac{c'^{2i+1}}{32^i}$ of the time in the first $n'_i$ symbols of $x$.  Then $\mu(C_{i+1}) \geq \min\{1,\mu(C_i)+\frac{3c'^2}{4}\}$.
\end{lem}

The proof of this lemma is similar to \cite{bosh_crit}.

\begin{proof}
We first show $\mu(C_{i+1}\setminus C_i)\geq \frac{8c'}{10}$ if $C_{i+1}^c\neq \emptyset$.

Let $u$ be a word of length $n'_{i+1}$ appearing in our system with the fewest occurrences of $1$; let $v$ be a word of length $n'_{i+1}$ with the most occurrences of 1. By our assumption that $C_{i+1}^c\neq \emptyset$, there are fewer than $\frac{c'^{2i+1}}{32^i}{n'_{i+1}}$ occurrences of 1 in $u$. Since $\epsilon_1 > c'$, there are at least $c' n'_{i+1}$ occurrences of 1 in $v$. Because $S':X'\to X'$ is minimal, there is a word $uwv=a_1,...,a_m$ occuring in $X'$. Let $j$ be the maximal index so that $\alpha_j:=a_j,...,a_{j+n'_{i+1}}$ has fewer than $\frac{c'^{2i+3}}{32^{i+1}}n'_{i+1}$ occurrences of the symbol 1; such an index exists by the remarks above. The cylinder set defined by $\alpha_\ell:=a_{\ell},..., a_{\ell+n'_{i+1}-1}$ is contained in $C_{i+1}$ for all $\ell>j$. 

We now estimate the proportion of length-$n'_i$ subwords of $\alpha_j$ which give cylinder sets in $C_i^c$. There are fewer than $\frac{c'^{2i+3}}{32^{i+1}}n'_{i+1}$ occurrences of 1 in $\alpha_j$, each of which occurs in at most $n'_i$ of its length-$n'_i$ subwords. Therefore, there are at most $\frac{c'^2}{32}n'_{i+1}$ length-$n'_i$ subwords (entirely) contained in $\alpha_j$ that give cylinders in $C_i$. There are $n'_{i+1}-n'_i+1$ total length-$n'_i$ subwords in $\alpha_j$. Therefore, we have at least $n'_{i+1}-n'_i + 1 - \frac{c'^2}{32}n'_{i+1}$ length-$n'_i$ subwords of $\alpha_j$ which give cylinders in $C_i^c$. All but perhaps the first length-$n'_{i+1}$ cylinder are in $C_{i+1}$. Using our assumption on $\epsilon_{n_{i+1}}$, this gives
\[ \mu(C_{i+1}\setminus C_i) \geq \left(n'_{i+1}-n'_i-\frac{c'^2}{32}n'_{i+1}\right) \frac{c'}{n'_{i+1}}.\]
Recalling that $n'_{i+1}>10n'_i$, the bound $\mu(C_{i+1}\setminus C_i) \geq \frac{8c'}{10}$ follows easily.
\\

Now we show that $\mu(C_i \setminus C_{i+1})\leq\frac{c'^2}{20}$. Let 
\[h_i:C_i \to \mathbb{N} \text{ by }h_i(x)=\min\{n>0:S^nx \in C_i\}.\] 
By the Kac Lemma (see for example \cite[Theorem 3.6]{kreng}) $\int _{C_i} h_id\mu=\mu(X)=1$. Let $u_x=\sum_{j=0}^{n'_{i+1}-n'_i}\chi_{C_i}(S'^jx)$ and suppose that $x \in C_i \setminus C_{i+1}$. Then  
\[u_x\leq \frac{c'^{2i+3}}{32^{i+1}}n'_{i+1}\left(\frac{c'^{2i+1}}{32^i}\right)^{-1} = \frac{c'^2}{32}n'_{i+1}.\] 
Indeed, there are fewer than $\frac{c'^{2i+3}}{32^{i+1}}n'_{i+1}$ occurrences of 1 in the word of length $n'_{i+1}$ corresponding to a point in $C_{i+1}$, each word giving a point in $C_i$ has at least $\frac{c'^{2i+1}}{32^i}n'_i$ occurrences of 1, and each occurrence of 1 appears in at most $n'_i$ different length-$n'_i$ words. 

Therefore, for each $x\in C_i\setminus C_{i+1}$, we have $\sum_{j=0}^{\frac{c'^2}{32}n'_{i+1}-1} h_i(S'|^j_{C_i} x) \geq \sum_{j=0}^{u_x} h_i(S'|_{C_i}^j x) \geq n'_{i+1}-n'_i$, where (as in Lemma \ref{lem:dense block}) $S'|_A$ denotes the first return map of $S'$ to $A$. Then
\begin{align}
	\frac{c'^2}{32}n'_{i+1} &= \int_{C_i} \sum_{j=0}^{\frac{c'^2}{32}n'_{i+1}-1} h_i(S'|^j_{C_i}x) d\mu \nonumber \\
			& \geq \int_{C_i\setminus C_{i+1}} \sum_{j=0}^{\frac{c'^2}{32}n'_{i+1}-1} h_i(S'|_{C_i}^jx) d\mu \nonumber \\
			& \geq (n'_{i+1}-n'_i) \mu(C_i\setminus C_{i+1}) \nonumber.
\end{align}
Then we have $\mu(C_i \setminus C_{i+1}) \leq (n'_{i+1}-n'_i)^{-1} \frac{c'^2}{32}n'_{i+1}$.  Since $n'_{i+1} \geq 10 n'_i$ a short calculation gives the bound $\mu(C_i\setminus C_{i+1}) \leq \frac{10}{9}\frac{c'^2}{32} \leq \frac{c'^2}{20}$.
\\

From these two bounds it follows that $\mu(C_{i+1}) \geq \mu(C_i) -\frac{c'^2}{20}+\frac{8c'}{10} \geq \mu(C_i)+\frac{3c'^2}{4}$.
\end{proof}

As a corollary we obtain:

\begin{cor}\label{cor:appearance} 
For any minimal symbolic system $S':X'\to X'$ with $\epsilon_1>c'$, $\epsilon_{n_i'}\geq \frac{c'}{n_i'}$ for a sequence $n'_i$ such that $n_i'\geq 10 n_{i-1}'$, there exists an integer $q'$ and a number $\delta>0$ (each depending only on $c'$) so that for any symbol $a$, any $x\in X'$ satisfies
\[ \sum_{i=1}^{n'_l} \chi_a(S'^ix) \geq \delta n'_l \]
for all $l\geq q'$. That is, at least a proportion $\delta$ of the first $n'_l$ symbols of $x$ are $a$'s. 
\end{cor}

\begin{proof}
Let $q'$ be such that $q'\frac{3c'^2}{4}\geq 1$. Let $\delta = \frac{c'^{2q'+1}}{32^{q'}}.$ For $l\geq q'$ and any $x$, each length-$n'_{q'}$ subword of the first $n'_l$ letters of $x$ has at least a proportion $\delta$ of the symbol $a$, by Lemma \ref{cliques} (applied for $a$ instead of 1). This establishes the Corollary. 
\end{proof}

We now describe a symbolic system describing the trajectory of points through the Rokhlin towers of Proposition \ref{towers}.

For any $n$, consider $\mathcal{R}_n=\{ (J_a,m_a)\}$, the set of Rokhlin towers given by Proposition \ref{towers} for this value of $n$. To a point $x\in [0,1)$ we assign the coding $\ldots a_0, a_1, a_2, \ldots$ if $x\in T^kJ_{a_0}$ for some $0\leq k <m_{a_0}$, $T^{n_{a_0}-k}x\in J_{a_1}$, $T^{n_{a_0}-k+n_{a_1}}x\in J_{a_2}$, and so on. In other words, $x$ begins in the $a_0$ tower, and subsequently visits the towers with indices $a_1, a_2, \ldots$.

 Let $X'_n$ be the set of such codings and $S'_n:X'_n\to X'_n$ the corresponding symbolic system. This system is topologically transitive since $T$ is, and it is an easy exercise to check that it satisfies $\epsilon_n > \frac{c'}{n}$ for $c'=\frac{c}{2}$. Apply Corollary \ref{cor:appearance} to this shift, using $n'_i = 10^i$, obtaining $q$ and $\delta$ which depend only on $c$ (and not on $n$). Without loss of generality, we assume $\delta<\frac{1}{3}$.

\begin{proof}[Proof of Theorem \ref{thm:quant bc} and Theorem \ref{thm:shift version}]

Let an integer $i$ and a word $w$ of length $n_i$ be given. We want to show that
\begin{equation}\label{eqn:bound goal}
	\frac{1}{\mu(w)}\sup_{x,x'} \frac{1}{n_{i+\hat{q}+L}}\left|\sum_{j=1}^{n_{i+\hat{q}+L}}\chi_\omega(S^j x)-\chi_\omega(S^j x')\right| < C_1 e^{-C_2L}
\end{equation}
for $C_1, C_2, \hat q>0$ depending only on $c$. We show this by first bounding $L=0$ case above with a bound depending only on $c$. Then we show that there is some $r>0$ depending only on $c$ such that the left-hand side of \eqref{eqn:bound goal} decays by a constant factor $\zeta<1$ depending only on $c$ for every increase of $r$ in $L$. These two facts will accomplish the proof.

\

$L=0$: We claim that there exist constants $\hat{q}, b, B>0$ depending only on $c$ with $\hat q\geq d(2-\log_2\xi)$, so that for all $x$,
\begin{equation}\label{eqn:base bounds}
	b \mu(w) \leq \frac{1}{n_{i+\hat q}}\sum_{j=1}^{n_{i+\hat{q}}}\chi_w(S^jx)\leq B \mu(w).
\end{equation}
From these bounds it will follow that for all $x, x'$,
\[ \frac{1}{n_{i+\hat q}} \left| \sum_{j=1}^{n_{i+\hat q}} \chi_w(S^jx) - \chi_w(S^j x') \right| <  (B-b)\mu(w) \]
as desired for the $L=0$ case. The lower bound will be used below in our proof of the exponential decay.

We prove the  upper bound in equation \eqref{eqn:base bounds}, with $B=\frac{2}{c}$, by an argument similar to that in Lemma \ref{lem:hit bound}. Partition the interval corresponding to $w$ into a minimal collection of subintervals of size $<\epsilon_{n_i}$. Since $\mu(w)\geq \epsilon_{n_i}$, there are at most $\lceil\frac{\mu(w)}{\epsilon_{n_i}}\rceil \leq 2\frac{\mu(w)}{\epsilon_{n_i}}$ of these. As $\epsilon_{n_i}>\frac{c}{n_i}$, this is $<\frac{2}{c}\mu(w)n_i$. These subintervals are hit at most once every $n_i$ iterates, so
\[ \frac{1}{n_{i+\hat q}}\sum_{j=1}^{n_{i+\hat q}} \chi_w(S^jx) \leq \frac{2}{c}\mu(w).\]
The choice of $\hat q$ is not relevant for this part of the argument.

For the lower bound in \eqref{eqn:base bounds} we do the following. Let $\hat n = n_{i+\log_{10}\frac{20}{c^2}}$. Note that $\hat n\geq \frac{20}{c^2} n_i$. Consider the set of towers $\mathcal{R}_{\hat n}$ given by Proposition \ref{towers}. Since the union of the towers is $[0,1)$, there exists some $a^*$ such that 
\[\mu((\cup_{i=0}^{m_{a^*-1}}S^iJ_{a^*})\cap w) \geq \mu(w) \mu(\cup_{i=0}^{m_{a^*-1}}S^iJ_{a^*})=\mu(w)m_{a^*}\mu(J_{a^*}) \geq \mu(w)\frac{c}{2}.\]
By construction, $\mu(S^iJ_{a^*})<\frac{1}{\hat n}$ and the $S^iJ_{a^*}$ are disjoint for $0\leq i\leq m_{a^*}$ so at least $\frac{\mu(w)c/2}{1/\hat n}=\hat n\mu(w)\frac{c}{2}\geq \frac{c}{4}m_{a^*}\mu(w)$ of them intersect $w$. The choice of $\hat n$ and the fact that $\mu(w)\geq \frac{c}{2n_i}$, imply $\frac{c}{4}m_{a^*}\mu(w)\geq 5$. At most two of these intersect $w$ in its boundary since $w$ codes for an interval in $[0,1)$. Therefore 
\[ \left|\{ 0 \leq i \leq m_{a^*}-1 : S^iJ_{a^*}\subseteq w\}\right| \geq \frac{c}{8}m_{a^*} \mu(w).\]

Apply Corollary \ref{cor:appearance} to the symbolic coding $S'_{\hat n}: X'_{\hat n} \to X'_{\hat n}$, obtaining $q'$ and $\delta$ so that for all $l\geq q'$, at least a proportion $\delta$ of any word of length $n'_l$ in $X'$ is the symbol $a^*$. Since the symbols in $X'_n$ correspond to words in $X$ of length $\leq 2\hat n$, every word of length $2\hat n(n'_{q'}+2)$ in $X$ contains a subword corresponding to a word of length $n'_{q'}$ in $X_{\hat n}'$ which accounts for at least a third of its length. Therefore, for all $x$
\[ \frac{1}{n_{i+\hat q}} \sum_{j=1}^{n_{i+\hat q} -1} \chi_w(S^jx) \geq \frac{1}{3}\delta\frac{c}{8} \mu(w) \]
where $\hat q = \log_{10}\frac{20}{c^2}+q'+1.$ We have the lower bound of \eqref{eqn:base bounds} with $b=\frac{c\delta}{24}$ and $\hat q$ depending only on $c$ by Corollary \ref{cor:appearance}.

\

\textit{Exponential decay:} 
We want to show that for all $x, \, x'$, and for all $i$ and $L$, there exists some $r>0$ and $\tilde{\zeta}<1$ such that
\begin{equation}\label{eq:induct decay}
	\frac{1}{n_{i+\hat{q}+L}}\left|\sum_{j=1}^{n_{i+\hat{q}+L}}\chi_{\omega}(S^jx)-\chi_\omega(S^jx')\right| \leq \mu(\omega)(B-b)\tilde{\zeta}^{\lfloor\frac L r\rfloor}.
\end{equation}

We will do this by showing that there exists $r$ so that for words formed by completely traversing $\mathcal{R}_{n_i+L+r}$ the maximum and minimum occurrences of $\omega$ differ by less than words formed by completely traversing $\mathcal{R}_{n_i+L}$. 
We split the general word in words formed by completely traversing $\mathcal{R}_{n_i+k\lfloor \frac L r \rfloor}$ as $k$ varies and a tiny leftover piece. Then a simple sublemma completes the proof.  

Let $u_x$ be the finite word in the $X'_{n_i+\hat q +L}$ coding which records the towers in $\mathcal{R}_{n_{i+\hat q+L}}$ traversed by $x$ over the orbit segment indexed by $[1,n_{i+\hat q+L+r}]$, omitting the first and last symbols, which correspond to towers which $x$ may not fully traverse. We choose $r$ so that $\frac{4}{10^r}<\frac{\delta^2}{2}$; note that it depends only on the $\delta$ given by Corollary \ref{cor:appearance}, and hence only on $c$.

Consider the towers in $\mathcal{R}_{n_{i+\hat q +L}}$ whose corresponding words in the coding $X$ have the maximal and minimal frequencies of $w$ as subwords. Denote these frequencies by $\mu(w)\Xi_L$ and $\mu(w)\xi_L$, respectively. By the argument using Corollary \ref{cor:appearance} which gave the lower bound of \eqref{eqn:base bounds} above, $\xi_L\geq b>0$ where $b$ depends only on $c$.

Now we apply Corollary \ref{cor:appearance} to the coding $X'_{n_{i+\hat q+L}}$, as defined in paragraphs between Corollary \ref{cor:appearance} and the start of the proof. In every word in $X'_{n_{i+\hat q+L}}$, each symbol appears with frequency $\geq \delta$ for any subword of length at least $q'$. In particular, this is true of the symbols $A$ and $a$ respectively representing the towers in which $w$ appears with frequencies $\mu(w)\Xi_L$ and $\mu(w)\xi_L$. Let us further assume that $r>q'$, a choice again depending only on $c$.

Now for our word $u_x\in X'_{n_{i+\hat q+L}}$, $A$ appears with frequency $\geq\delta$ and $a$ appears with frequency $\geq\delta$. Therefore, the frequency of $w$ for $x$ is between $\delta\mu(w)\Xi_L+(1-\delta)\mu(w)\xi_L$ and $(1-\delta)\mu(w)\Xi_L+\delta\mu(w)\xi_L$ (up to the small error coming from omitting the initial and final symbols in forming $u_x$). The frequency of $w$ for $x$ thus lies in a range of size bounded above by $(1-2\delta)\mu(w)(\Xi_L-\xi_L)$. Letting $\zeta = 1-2\delta$ proves \eqref{eq:induct decay} for those indices $j$ covered by the word $u_x$. (Our eventual $\zeta$ will be different.)

Recall that for any $x\in X$ we have written the orbit of $x$ over the indices $[1,n_{i+\hat q+L+r}]$ as a prefix of length $\leq 2n_{i+\hat q +L}$, a core piece during which the orbit fully traverses towers from $\mathcal{R}_{n_{i+\hat q +L}}$ and then a suffix of length $\leq 2n_{i+\hat q +L}$. The work above shows that range of frequencies with which core piece of the orbit hits the cylinder set defined by $w$ decays by the factor $\zeta<1$ each time $L$ increases by $1$. To complete the proof we need to incorporate the prefix and suffix.

Note that the prefix and suffix take up a proportion $\leq \frac{4}{10^r}$ of the indices in $[1,n_{i+\hat q+L+r}]$. Decompose the prefix and suffix into core orbit segments fully traversing towers from $\mathcal{R}_{n_{i+\hat q +L-r}}$, leaving a second set of prefix and suffix segments each of length at most $ n_{i+\hat q +L-r}$. These segments take up a proportion $\leq 2 \cdot \frac{4^2}{10^{2r}}$ of the indices in $[1,n_{i+\hat q+L+r}]$ and the range of frequencies with which $w$ appears in these segments is bounded above by $\mu(w)(\Xi_{L-r} - \xi_{l-r}) < \zeta \mu(w)(\Xi_L-\xi_L)$. Proceeding in this way, we decompose the original suffix and prefix into segments taking up a proportion $\leq 2^k \cdot \frac{4^k}{10^{kr}}$ of the indices in $[1,n_{i+\hat q+L+r}]$ in which $w$ appears with a frequency range $<\mu(w)(\Xi_{L-kr} - \xi_{l-kr}) < \zeta \mu(w)(\Xi_{L-(k-1)r}-\xi_{L-(k-1)r})$.

Therefore, we can bound the range of frequencies for the entire prefix and suffix by $\mu(w)\sum_{k=1}^{L-1}\frac{4^k}{10^{kr}}(\Xi_{L-kr}-\xi_{L-kr})$.  Then the total frequency of $w$ over the indices $[1, n_{i+\hat q +L+r}]$ lies in a range of size bounded above by 
\begin{equation}\label{eq:final}
	(1-2\delta)\mu(w)(\Xi_L-\xi_L)+\mu(w)\sum_{k=1}^{L-1}\frac{4^k}{10^{kr}}(\Xi_{L-kr}-\xi_{L-kr}).
\end{equation}

To prove exponential decay we use the following fact:

\begin{slem}
If $0<\alpha,\gamma<1$ satisfy $\alpha+\gamma \alpha^{-1}<1$ and $(x_j)$ is a sequence of positive numbers so that $x_{j+1}<\alpha x_j+\sum_{k=1}^{j-1} \gamma^k x_{j-k}$ for all $j$, then $x_j<x_0 (\alpha + \gamma \alpha^{-1})^j$ for all $j$. 
\end{slem}

\begin{proof}[Proof of Sublemma] The proof is by induction. The $j=0$ case is immediate. Then
$$x_{j+1}<\alpha x_j+\sum_{k=1}^{j-1}\gamma^kx_{j-k}<\alpha (\alpha+\gamma\alpha^{-1})^jx_0+\sum_{k=1}^{j-1}\gamma^k(\alpha+\gamma \alpha^{-1})^{j-k}x_0,$$
where the final inequality is by induction. Now  
\begin{align}
	\sum_{k=1}^{j-1}\gamma^k(\alpha+\gamma \alpha^{-1})^{j-k} &= (\alpha+\gamma\alpha^{-1})^j\sum_{k=1}^{j-1}(\frac{\gamma}{\alpha+\gamma\alpha^{-1}})^k \nonumber \\
		&<(\alpha+\gamma\alpha^{-1})^j \frac{\gamma}{\alpha+\gamma \alpha^{-1}}\left(1-\frac{\gamma}{\alpha+\gamma\alpha^{-1}}\right)^{-1} \nonumber \\
		& = (\alpha+\gamma\alpha^{-1})^j \frac{\gamma}{\alpha+\gamma\alpha^{-1}-\gamma} \nonumber \\
		& < (\alpha+\gamma\alpha^{-1})^j \frac{\gamma}{\alpha}. \nonumber
\end{align}
Therefore, 
\[ x_{j+1} < \left[ \alpha(\alpha+\gamma\alpha^{-1})^j+(\alpha+\gamma\alpha^{-1})^j \frac{\gamma}{\alpha}\right]x_0\]
which simplifies to the desired result.
\end{proof}

Let $x_j=\Xi_{jr}-\xi_{jr}$, $\alpha = 1-2\delta$ and $\gamma=\frac{4}{10^r}< \frac{\delta^2}{2}$. Then $\hat \zeta = (1-2\delta + \frac{4}{10^r(1-2\delta)}) <(1-2\delta)+\frac{\delta^2}{2}\delta^{-1}<1$ (using $\delta <\frac{1}{3}$). 
Applying the Sublemma, using the $L=0$ case to bound $x_0$ by $\mu(\omega)(B-b)$ we bound \eqref{eq:final}  by 
$$\mu(\omega)(B-b)\left(1-2\delta+\frac 4 {10^r(1-2\delta)}\right)^{\lfloor\frac L r\rfloor},$$
implying the theorem.


\end{proof}

%

 \appendix
 \section{Symbolic coding for IETs}\label{symb code}

We use the symbolic coding of interval exchange transformations and concepts related to it. In this Appendix we supply some standard definitions and terminology related to this coding. We show the well known and useful fact that IETs are basically the same as (measure conjugate to) continuous maps on compact metric spaces, and we recall the definition of a Rokhlin tower, a concept which appears in the proof of Theorem \ref{thm:quant bc}.

\begin{defin}[Standard coding for an IET]
The \emph{standard coding} of an interval exchange transformation $T$ with intervals $I_i$ is given by
\begin{center} $\tau \colon [0,1) \to \{1,2,...,d\}^{\mathbb{Z}}$ by $\tau(x)=...,a_{-1},a_0,a_1,...$ where $T^i(x) \in I_{a_i}$. \end{center}
\end{defin} 

Note that the coding map $\tau$ is not continuous as a map from $[0,1)$ with the standard topology to $\{1,2,...,d\}^{\mathbb{Z}}$ with the product topology.

\begin{defin}[Blocks of a coding]
Fix a point $x$, that is not in the orbit of a discontinuity of $T$. Let
\begin{center} $w_{p,q}(x)=c_p,c_{p+1},...,c_{q-1},c_q$ where $\tau(x)=...c_{-1},c_{0},c_1,...$ \end{center} 
This word is a \emph{block} of \emph{length} $q-p$, or a $(q-p)$\emph{-block}.
\end{defin}

A key element in our proof is the $n$-block interval:

\begin{defin}[$n$-block interval]\label{defn:n block interval}
An interval $J\subset [0,1)$ is an \emph{$n$-block interval} if $J=\{x: w_{0,n}(x) = w_{0,n}(x_0) \mbox{ for some } x_0\}$.
\end{defin}
Note that the \emph{measure} of an $n$-block interval is the length of the interval $J$. We use `measure' rather than `length' so as not to create confusion with the length $n$ of the coding block corresponding to this $n$-block interval.

We would like to consider $\tau([0,1))$ as a subshift of the full shift on $\{1,\ldots d\}^\mathbb{Z}$, but the situation is not so simple. Observe that the left shift $S$ acts continuously on $\tau([0,1))\subset \{1,2,...,d\}^{\mathbb{Z}}$. However, if $T$ satisfies the Keane condition, then $\tau([0,1))$ is not closed in $\{1,2,...,d\}^{\mathbb{Z}}$ with the product topology. To see this, consider points just to the left of a discontinuity of $T$ and the $n$-blocks $w_{0,n}(x)$ corresponding to them. As $x$ approaches the discontinuity and $n\to \infty$, these finite blocks do not converge to an infinite block in $\tau([0,1))$. Let $ \hat X$ be the closure of  $\tau([0,1))$  in $\{1,2,...,d\}^{\mathbb{Z}}$ with the product topology. $\hat X$ results from  adding a countable number of points to $\tau([0,1))$ which correspond to the left hand sides of points in orbits of a discontinuity. $\hat X$ is a compact metric space and, equipped with the left shift $S$, is a subshift. Equip $\hat X$ with a measure $\mu$ assigning to the cylinder set defined by each block the Lebesgue measure of the corresponding block interval in $[0,1)$.

 Let $f: \hat X \to [0,1)$ by $f|_{\tau([0,1))}= \tau^{-1}$ and extend $f$ by continuity to the rest of $\hat X.$ 
 Notice that, unlike $\tau$, the map $f$ \emph{is} continuous. 
  Moreover the map is injective away from $\tau^{-1}$ of the orbits of discontinuities, where it is $2$-to-$1$. The left shift $S$ acts continuously on $\hat X$ and if $T$ satisfies the Keane condition, then the action of $S$ on $(\hat X,\mu)$ is measure conjugate to the action of  $T$ on $([0,1), Leb)$.

\begin{defin}[Rokhlin Tower]
Let half open intervals $J_1,...,J_r$ and natural numbers $m_1,...,m_r$ be given such that 
\begin{itemize}
	\item $T^j$ is continuous (thus an isometry) on $J_i$ for $0\leq j < m_i$,   
	\item $\underset{i=1}{\overset{r}{\cup}}\underset{j=0}{\overset{m_i-1}{\cup}}T^j(J_i)=[0, 1)$, and 
	\item $T^j(J_i) \cap T^{j'}(J_{i'})=\emptyset$ when $0 \leq j <j' < m_i$, $0\leq j' < m_{i'}$ and  $j\neq j'$ if $i=i'$.
\end{itemize}
Then we say that the $\underset{j=0}{\overset{m_i-1}{\cup}}T^j(J_i)$ are \emph{Rokhlin towers}. $m_i$ is called the \emph{height} of the Rokhlin tower. Each $T^j(J_i)$ is called a \emph{level} of the tower. 
\end{defin}

Rokhlin towers and the symbolic coding are closely related. Up to a suffix and a prefix, every word in $\tau([0,1))$ is a concatenation of the length $m_i$ coding of the points in $J_i$ as $i$ ranges in $\{1,...,r\}$. The prefix and suffix are subwords of these codings.


\begin{thebibliography}{xxx}
\bibitem{ath_forn}  Athreya, J. S.; Forni, G: Deviation of ergodic averages for rational polygonal billiards. \textit{Duke Math. J.} 144 (2008), no. 2, 285--319.
\bibitem{dukeJ} Boshernitzan, M: \textit{A condition for minimal interval exchange maps to be uniquely ergodic}. Duke Math. J. 52 (1985), no. 3, 723-752. 
\bibitem{rank2} Boshernitzan, M: \textit{Rank two interval exchange transformations.} {Ergod. Th. \& Dynam. Sys.}  \textbf{8}  (1988),  no. 3, 379--394.
\bibitem{bosh_crit} Boshernitzan: \textit{A condition for unique ergodicity of minimal symbolic flows.} Erg. Th. \& Dynam. Sys. \textbf{12} (1992), no. 3, 425-428. 
\bibitem{gauges} Boshernitzan, M; Chaika, J: Quantitative proximality and connectedness.  \textit{Inventiones} 192 (2013), no. 2, 375--412.
\bibitem{iet kurz} Chaika, J: \textit{Shrinking targets for IETs: Extending a theorem of Kurzweil}. \textit{Geom. Func. Anal.} 21 (2011), no. 5, 1020-104.
\bibitem{CCM} Chaika, J; Cheung, Y; Masur, H: \textit{Winning games for bounded geodesics on Teichmueller discs.} arxiv:1109.5976
\bibitem{sbc} Chernov, N; Kleinbock, D:
\textit{Dynamical Borel-Cantelli lemmas for Gibbs measures}.
Israel J. Math. \textbf{122} (2001) 1-27. 
\bibitem{dolg} Dolgopyat, D: \textit{Limit theorems for partially hyperbolic systems.} Trans. AMS \textbf{356} (2004) no. 4, 1637-1689.
\bibitem{EM} Eskin, A; Masur, H: \textit{Asymptotic formulas on flat surfaces}. Erg. Th. \& Dynam. Sys. \textbf{21} (2001), no. 2, 443-478.
\bibitem{forni}Forni, G: Deviation of ergodic averages for area-preserving flows on surfaces of higher genus. \textit{Ann. of Math.} (2) 155 (2002), 1--103.
\bibitem{galatolo} Galatolo, S: Hitting time and dimension in axiom A systems, generic interval exchanges and an application to Birkoff sums. \textit{ J. Stat. Phys.}  \textbf{123}  (2006),  no. 1, 111--124.
\bibitem{ker} Kerckhoff, S. P: Simplicial systems for interval exchange maps and measured foliation. Ergod. Th. \& Dynam. Sys. \textbf{5} (1985),  257-271.
\bibitem{kesten}  Kesten, H:  \textit{On a conjecture of Erd\"{o}s and Sz\H{u}z related 
   to uniform distribution mod 1}. 
  Acta Arith.
   \textbf{12} (1966), 193--212
\bibitem{khinchin} Khinchin: \textit{Continued fractions}. Dover
\bibitem{log iet} Kim, D.H.; Marmi, S: The recurrence time for interval exchange maps.  \textit{Nonlinearity}  \textbf{21}  (2008),  no. 9, 2201--2210.\bibitem{baflat} Kleinbock, D, Weiss, B: \textit{Bounded geodesics in moduli space}. Int. Math. Res. Not. 2004, no. 30, 1551-1560.
\bibitem{kreng} Krengel, U: \textit{Ergodic Theorems} Walter de Gruyter \& Co.
\bibitem{kurz} Kurzweil, J: \textit{On the metric theory of inhomogeneous diophantine approximation}. Studia. Math. \textbf{15} (1955) 84-112.
\bibitem{march} Marchese, L: The Khinchin theorem for interval exchange transformations. \textit{J. Mod. Dyn.}   2011 no. 1, 123-183.
\bibitem{mmy} Marmi, S; Moussa, P; Yoccoz, J.-C: The cohomological equation for Roth-type interval exchange maps. \textit{J. Amer. Math. Soc.} 18 (2005), no. 4, 823--872.
\bibitem{masur crit} Masur, H: \textit{Hausdorff dimension of the set of nonergodic foliations of a quadratic differential}. Duke Math. J. 66 (1992), no. 3, 387-442.
\bibitem{t surf} Masur, H: \textit{Ergodic theory of translation surfaces}.  Handbook of dynamical systems. Vol. 1B,  527--547, Elsevier B. V., Amsterdam, 2006.
\bibitem{McM} McMullen, C: \textit{ Winning sets, quasiconformal maps and Diophantine approximation}. Geom. Funct. Anal.  2010, {\bf 20} 726-740,
\bibitem{philipp} Philipp, W: \textit{Some metrical theorems in number theory}. Pac. J. Math. \textbf{20} (1967), 109-127.
\bibitem{schmidt}Schmidt, W.: \textit{ A metrical theorem in diophantine approximation}. Canad. J. Math. \textbf{12} 1960 619-631.
\bibitem{V bosh crit} Veech, W: \textit{ Boshernitzan's criterion for unique ergodicity of an interval exchange transformation}. Erg. Th. \& Dynam. Sys. \textbf{7} (1987), no. 1, 149-153.
\bibitem{zorich} Zorich, A: Deviation for interval exchange transformations. \textit{Erg. Th. \& Dynam. Sys.} 17 (1997), 1477--1499.
\bibitem{survey} Zorich, A: \textit{Flat surfaces}.  Frontiers in number theory, physics, and geometry. I,  437--583, Springer, Berlin, 2006.
\end{thebibliography}
\end{document}